\documentclass[a4paper]{amsart}                

\usepackage[latin1]{inputenc}                  
\usepackage[T1]{fontenc}                       
\usepackage[spanish,english]{babel}            
\usepackage[pdftex]{graphicx}                  
\usepackage{amsmath,amssymb,amsthm}            
\usepackage{latexsym}                          
\usepackage{delarray}                          
\usepackage{bbm}                               
\usepackage{hyperref}                          
\usepackage[pdftex,usenames,dvipsnames]{color} 
\usepackage{bbding,trfsigns}                   
\usepackage{wasysym}                           
\usepackage{datetime}                          

\DeclareMathAlphabet{\mathpzc}{OT1}{pzc}{m}{it} 
\newcommand{\B}{\mathbb{B}}
\newcommand{\R}{{\mathbb{R}^n}}

\newtheorem*{ThSN}{Theorem}
\newtheorem{Th}{Theorem}[section]              

\newtheorem{Prop}{Proposition}[section]
\newtheorem{Lem}{Lemma}[section]

\title[Square functions for Schr\"odinger and Laguerre operators in UMD spaces]
      {UMD Banach spaces and square functions associated with heat semigroups for Schr\"odinger and Laguerre operators}
\author{J.J. Betancor}
\author[A.J. Castro]{A.J. Castro}
\author[J.C. Fari\~{n}a]{J.C. Fari\~{n}a}
\author[L. Rodr\'{\i}guez-Mesa]{L. Rodr\'{\i}guez-Mesa}
\address{\newline
        Jorge J. Betancor, Alejandro J. Castro, Juan C. Fari\~na and Lourdes Rodr\'{\i}guez-Mesa \newline
        Departamento de An\'alisis Matem\'atico,
        Universidad de La Laguna, \newline
        Campus de Anchieta, Avda. Astrof\'{\i}sico Francisco S\'anchez, s/n, \newline
        38271, La Laguna (Sta. Cruz de Tenerife), Spain}
\email{jbetanco@ull.es, ajcastro@ull.es, jcfarina@ull.es, lrguez@ull.es}


\begin{document}

\footnotetext{Date: \today.}

\maketitle                                  

\begin{abstract}
    In this paper  we define square functions (also called Littlewood-Paley-Stein functions) associated with heat semigroups
    for Schr\"odinger and Laguerre operators acting on functions which take values in UMD Banach spaces. We extend classical (scalar)
    $L^p$-boundedness properties for the square functions to our Banach valued setting by using $\gamma$-radonifying operators.
    We also prove that these $L^p$-boundedness properties of the square functions actually characterize the Banach spaces having the UMD property.
\end{abstract}

\section{Introduction}

Suppose that $(\Omega , \mu)$ is a measure space and $\{T_t\}_{t>0}$ is  an analytic semigroup on $L^p(\Omega , \mu)$, where $1\leq p\leq \infty$.
If $k\in \mathbb{N}$, the $k$-th vertical square function $g^k(\{T_t\}_{t>0})(f)$ of $f\in L^p(\Omega ,\mu )$ is defined by
$$g^k(\{T_t\}_{t>0})(f)(x)
    =\left(\int_0^\infty \left|t^k\partial_t^kT_t(f)(x)\right|^2\frac{dt}{t}\right)^{1/2}.$$
The $L^p$-boundedness properties of $g^k$-square functions are very useful in order to describe the behavior in $L^p$- spaces of multipliers associated
to the infinitesimal generator of the semigroup $\{T_t\}_{t>0}$ (see \cite{Me}, \cite{Ste1} and \cite{Th}).

It is well-known (\cite[p. 120]{Ste1}) that if $\{T_t\}_{t>0}$ is the classical
heat or Poisson semigroup then,
for every $1<p<\infty$,
\begin{equation}\label{gequivf}
    \|g^k(\{T_t\}_{t>0})(f)\|_{L^p(\mathbb{R}^n)}
        \sim \|f\|_{L^p(\mathbb{R}^n)},\quad f\in L^p(\mathbb{R}^n).
\end{equation}
This property can be extended to other semigroups of operators (see \cite{MTX}, \cite{Ste1}, \cite{Th}, \cite{Xu}, amongst others).

In the sequel we denote as usual by $\{W_t\}_{t>0}$ and $\{P_t\}_{t>0}$ the classical heat and Poisson semigroup on $\mathbb{R}^n$, respectively.
We have that, for every $t>0$ and $f\in L^p(\mathbb{R}^n)$, $1\leq p\leq \infty$,
$$W_t(f)(x)
    =\int_{\mathbb{R}^n}\frac{e^{-|x-y|^2/(4t)}}{(4\pi t)^{n/2}}f(y)dy,\quad x\in \mathbb{R}^n,$$
and
$$P_t(f)(x)
    =c_n\int_{\mathbb{R}^n}\frac{t}{(t^2+|x-y|^2)^{(n+1)/2}}f(y)dy,\quad x\in \mathbb{R}^n,$$
being $c_n= \pi^{-(n+1)/2} \Gamma((n+1)/2)$.

If $\psi :\mathbb{R}^n\longrightarrow \mathbb{R}$ is a measurable function on $\mathbb{R}^n$, we define $\psi _t(x)=t^{-n}\psi (x/t)$, $x\in \mathbb{R}^n$
and $t>0$. Then, it is clear that, for every $t>0$, $W_t(f)=G_{\sqrt{t}}*f$ and $P_t(f)=P_t*f$,
where $G(x)=\frac{e^{-|x|^2/4}}{(4\pi)^{n/2}}$, $x\in \mathbb{R}^n$, and $P(x)=\frac{c_n}{(1+|x|^2)^{(n+1)/2}}$, $x\in \mathbb{R}^n$.
We can also write, for every $k\in \mathbb{N}$,
$$g^k(\{W_t\}_{t>0})(f)(x)
    =\sqrt{2}\|\varphi ^k_{\sqrt{t}}*f(x)\|_{L^2\left((0,\infty ),\frac{dt}{t}\right)},$$
where $\varphi ^k(x)=\Big(\partial_t^kG_{\sqrt{t}}(x)\Big)_{\big|t=1}$, $x\in \mathbb{R}^n$, and
$$g^k(\{P_t\}_{t>0})(f)(x)
    = \|\phi^k_t*f\|_{L^2\left((0,\infty ),\frac{dt}{t}\right)},$$
where $\phi^k(x)=\Big(\partial_t^k P_t (x)\Big)_{\big|t=1}$, $x\in \mathbb{R}^n$.

If $\psi$ is good enough the continuous $\psi$-wavelet transform $\mathcal{W}_\psi (f)$ of $f\in L^p(\mathbb{R}^n)$, $1\leq p\leq \infty$, is defined by
$$\mathcal{W}_\psi (f)(x,t)
    =(\psi _t*f)(x),\quad x\in\mathbb{R}^n\mbox{ and }t>0.$$
In \cite{ChFS} (see also \cite{FS}) the authors gave conditions on the function $\psi$ so that the equivalence
\begin{equation}\label{Wequivf}
    \|\mathcal{W}_\psi (f)\|_{L^p\left(\mathbb{R}^n,L^2((0,\infty), \frac{dt}{t})\right)}
        \sim \|f\|_{L^p(\mathbb{R}^n)},
\end{equation}
holds for every $f\in L^p(\mathbb{R}^n)$, $1<p<\infty$. Note that \eqref{Wequivf} can be seen as an extension of \eqref{gequivf}
for the classical heat and Poisson semigroups.

In the last years several  authors (\cite{HTV}, \cite{Hy}, \cite{HNP}, \cite{K}, \cite{KW1}, \cite{MTX} and \cite{Xu})
have dealt with square functions acting on functions which take values in a Banach space.
Suppose that $\mathbb{B}$ is a Banach space and $f:\Omega \longrightarrow \mathbb{B}$ is a $\mu$- strongly measurable function.
The first (and maybe the more natural) definition of $g^k_\mathbb{B}(\{T_t\}_{t>0})(f)$ is the following:
$$g^k_\mathbb{B}(\{T_t\}_{t>0})(f)(x)
    =\left(\int_0^\infty \left\|t^k\partial_t^kT_t(f)(x)\right\|_\mathbb{B}^2 \frac{dt}{t}\right)^{1/2}.$$
This $g^k_\mathbb{B}$-square function was studied for the classical Poisson semigroup on the torus by Xu (\cite{Xu});
for the Poisson semigroup defined by the Ornstein-Uhlenbeck semigroup by Harboure, Torrea and Viviani (\cite{HTV});
for subordinated Poisson semigroups of diffusion semigroups (in the sense of Stein \cite{Ste1})  by Mart\'{\i}nez, Torrea and Xu (\cite{MTX});
and for Poisson semigroups associated with Schr\"odinger operators by Torrea and Zhang (\cite{TZ}). From the results in \cite{MTX} and \cite{Xu}
we can deduce the following.
\begin{ThSN}
    Let $\mathbb{B}$ be a Banach space and $1<p<\infty$. Then, the following assertions are equivalent.

    $(i)$ $\mathbb{B}$ is isomorphic to a Hilbert space.

    $(ii)$ For every $f\in L^p(\mathbb{R}^n,\mathbb{B})$,
    $$\|g_\mathbb{B}^1(\{P_t\}_{t>0})(f)\|_{L^p(\mathbb{R}^n)}
        \sim \|f\|_{L^p(\mathbb{R}^n,\mathbb{B})}. $$
\end{ThSN}
Other authors (\cite{Hy}, \cite{HNP}, \cite{K} and \cite{KW1}) have extended the definition of the $g$-square functions to a
Banach valued setting by different points of view. As one of their goals, they wanted to extend the equivalence in \eqref{gequivf} to
Banach spaces which are not isomorphic to Hilbert spaces. Hyt\"onen \cite{Hy} extended \eqref{gequivf} to a UMD Banach space setting by using
Banach-valued stochastic integration. On the other hand, Kaiser and Weis \cite{KW1} generalized \eqref{Wequivf} to functions taking values
in UMD Banach spaces by using $\gamma$-radonifying operators. These two approaches are closely connected
(see, for instance, \cite{NVW} and \cite{NeWe}).
In this paper we use $\gamma$-radonifying operators to study $g$-square functions associated with the heat semigroups for Schr\"odinger
and Laguerre operators in UMD Banach spaces.

The main properties of UMD Banach spaces can be encountered in \cite{Bou}, \cite{Bu3} and \cite{Rub}.

Suppose that $H$ is a separable Hilbert space and $\mathbb{B}$ is a real Banach space.
We take a sequence $(\gamma _k)_{k\in \mathbb{N}}$ of independent standard Gaussians.
We say that an operator $T$ bounded from $H$ into $\mathbb{B}$, shortly $T\in L(H,\mathbb{B})$, is $\gamma$-radonifying,
written $T\in \gamma (H,\mathbb{B})$, when
$$\|T\|_{\gamma (H,\mathbb{B})}
    =\left(\mathbb{E}\left\|\sum_{k=1}^\infty \gamma _kT(h_k)\right\|_\mathbb{B}^2\right)^{1/2}<\infty,$$
where $\{h_k\}_{k\in \mathbb{N}}$ is an orthonormal basis in $H$. If $\mathbb{B}$ is a Banach space not containing a copy of $c_0$ (that is the case of UMD spaces), then
\begin{equation}\label{normsup}
\|T\|_{\gamma (H,\mathbb{B})}
    =\sup \left(\mathbb{E}\left\|\sum_{k=1}^\infty \gamma _kT(h_k)\right\|_\mathbb{B}^2\right)^{1/2},
\end{equation}
where the supremum is taken over all the finite families $\{h_k\}$ of orthonormal functions in $H$ (\cite[Theorem 5.9]{Nee}). 
In the sequel by $H$ we denote the space $L^2((0,\infty ),dt/t)$.\\

If $f:(0,\infty )\longrightarrow \mathbb{B}$ is a strongly $\mu$-measurable function such that,
for every $L\in \mathbb{B}^*$, $L\circ f\in H$, then there exists $T_f\in L(H,\mathbb{B})$ such that
$$\langle L,T_f(h)\rangle
    =\int_0^\infty \langle L,f(t)\rangle_{\B^*,\B} h(t)\frac{dt}{t},\quad h\in H\mbox{ and }L\in \mathbb{B}^*.$$
We say that $f\in \gamma ((0,\infty ),dt/t,\mathbb{B})$ provided that $T_f\in \gamma (H,\mathbb{B})$.
We identify $f$ with $T_f$.
If $\mathbb{B}$ does not contain a copy of $c_0$ then
$\gamma ((0,\infty ),dt/t,\mathbb{B})$ is a dense subspace of $\gamma (H,\mathbb{B})$ (\cite[Remark 2.16]{KW1}). 
In the sequel we assume that $\mathbb{B}$ is UMD. Then, $\mathbb{B}$ does not contain a copy of $c_0$.

In \cite[Theorem 4.2]{KW1} Kaiser and Weis gave conditions over the function $\psi$ in order to the wavelet transform
$\mathcal{W}_\psi$ satisfies the following equivalence:
\begin{equation}\label{WequivfB}
    \|\mathcal{W}_\psi (f)\|_{L^p(\mathbb{R}^n,\gamma (H,\mathbb{B}))}\sim \|f\|_{L^p(\mathbb{R}^n,\mathbb{B})},
\end{equation}
for every $f\in L^p(\mathbb{R}^n,\mathbb{B})$ and $1<p<\infty$. Note that, since $\gamma (H,\mathbb{C})=H$, \eqref{WequivfB} reduces
to \eqref{Wequivf} when $\mathbb{B}=\mathbb{C}$. Then, \eqref{WequivfB} can be seen as an extension of \eqref{gequivf} when we consider
the classical heat or Poisson semigroups and functions taking values in a UMD Banach space.

In this paper we extend the equivalence \eqref{gequivf} to a UMD-Banach valued setting for the heat semigroup defined by
Schr\"odinger operator in $\mathbb{R}^n$, $n\geq 3$, the Hermite operator on $\mathbb{R}^n$, $n\geq 1$, and the Laguerre operator on $(0,\infty)$.
Then, we prove that these new equivalences allow us to characterize the UMD Banach spaces.

The Schr\"odinger operator $\mathcal{L}$ is defined by $\mathcal{L}=-\Delta +V$ in $\mathbb{R}^n$, $n\geq 3$,
where $\Delta$ is the Euclidean Laplacian in $\mathbb{R}^n$ and $V$ is a nonnegative measurable function in $\mathbb{R}^n$.
Here we assume that $V\in RH_s(\mathbb{R}^n)$, that is, $V$ satisfies the following $s$-reverse H\"older's inequality: there exists
$C>0$ such that, for every ball $B$ in $\mathbb{R}^n$,
\begin{equation}\label{RH}
    \left(\int_BV(x)^sdx\right)^{1/s}\leq C\int_BV(x)dx,
\end{equation}
where $s>n/2$.
If $E_\mathcal{L}$ represents the spectral measure associated with the operator $\mathcal{L}$, the heat semigroup of operators generated by
$-\mathcal{L}$ is denoted by $\{W_t^\mathcal{L}\}_{t>0}$, where
$$W_t^\mathcal{L}(f)
    =\int_{[0,\infty )}e^{-\lambda t}E_\mathcal{L}(d\lambda )f,\quad f\in L^2(\mathbb{R}^n).$$
We can write, for every $f\in L^2(\mathbb{R}^n)$,
\begin{equation}\label{WSchrodinger}
    W_t^\mathcal{L}(f)(x)=\int_{\mathbb{R}^n}W_t^\mathcal{L}(x,y)f(y)dy,\quad x\in \mathbb{R}^n\mbox{ and }t>0.
\end{equation}
The main properties of the kernel function $W_t^\mathcal{L}(x,y)$, $t>0$, $x,y\in \mathbb{R}^n$, can be encountered in
\cite{DGMTZ} and \cite{Sh}. Also, for every $t>0$, the operator $W_t^\mathcal{L}$ defined in \eqref{WSchrodinger} is bounded from
$L^p(\mathbb{R}^n)$ into itself, $1\leq p\leq \infty$. Thus, $\{W_t^\mathcal{L}\}_{t>0}$ is a positive semigroup of contractions in
$L^p(\mathbb{R}^n)$, $1\leq p\leq \infty$.

The Hermite (also called harmonic oscillator) operator $\mathcal{H}=-\Delta +|x|^2$ is a special case of the Schr\"odinger operator.
Here we consider $\mathcal{H}$ on $\mathbb{R}^n$, with $n\geq 1$. We define, for every $k\in \mathbb{N}$, the $k$-th  Hermite function ${\mathfrak{h}}_k$ by
$${\mathfrak{h}}_{k}(x)
    =(\sqrt{\pi}2^kk!)^{-1/2}e^{-x^2/2}H_k(z),\quad x\in\mathbb{R},$$
where by $H_k$ we denote the $k$-th Hermite polynomial (\cite[pp. 105--106]{Sz}).
If $k=(k_1,...,k_n)\in \mathbb{N}^n$ the $k$-th Hermite function ${\mathfrak{h}}_k$ is defined by
$${\mathfrak{h}}_k(x)
    =\prod_{j=1}^n{\mathfrak{h}}_{k_j}(x_j), \quad x=(x_1,...,x_n)\in \mathbb{R}^n.$$
The system  $\{{\mathfrak{h}}_k\}_{k\in \mathbb{N}^n}$ is orthonormal and complete in $L^2(\mathbb{R}^n)$.
Moreover, $\mathcal{H}{\mathfrak{h}}_k=(2|k|+n){\mathfrak{h}}_k$, where $|k|=k_1+...+k_n$ and $k=(k_1,...,k_n)\in \mathbb{N}^n$.
The operator $-\mathcal{H}$ generates in $L^2(\mathbb{R}^n)$ the semigroup of operators $\{W_t^\mathcal{H}\}_{t>0}$ where, for every $t>0$,
$$W_t^\mathcal{H}(f)
    =\sum_{k\in \mathbb{N}^n}e^{-t(2|k|+n)}c_k(f){\mathfrak{h}}_k,\quad f  \in L^2(\mathbb{R}^n),$$
being
$$c_k(f)
    =\int_{\mathbb{R}^n}{\mathfrak{h}}_k(y)f(y)dy,\quad k\in \mathbb{N}^n\mbox{ and }f\in L^2(\mathbb{R}^n).$$
According to the Mehler's formula (\cite[(1.1.36)]{Th}) we can write, for every $t>0$,
\begin{equation}\label{WHermite}
    W_t^\mathcal{H}(f)(x)=\int_{\mathbb{R}^n}W_t^\mathcal{H}(x,y)f(y)dy,\quad f\in L^2(\mathbb{R}^n,\mathbb{B}),
\end{equation}
where, for each $x,y\in \mathbb{R}^n$ and $t>0$,
$$W_t^\mathcal{H}(x,y)
    =\frac{1}{\pi ^{n/2}}\left(\frac{e^{-2t}}{1-e^{-4t}}\right)^{n/2}\exp\left[-\frac{1}{4}\left(|x-y|^2\frac{1+e^{-2t}}{1-e^{-2t}}+|x+y|^2\frac{1-e^{-2t}}{1+e^{-2t}}\right)\right].$$
By defining $W_t^\mathcal{H}$, for every $t>0$, on $L^p(\mathbb{R}^n)$, $1\leq p\leq \infty$, by means of \eqref{WHermite},
then the system $\{W_t^\mathcal{H}\}_{t>0}$ is a positive semigroup of contractions in $L^p(\mathbb{R}^n)$, $1\leq p\leq \infty$.

Since $\{W_t^\mathcal{L}\}_{t>0}$ and $\{W_t^\mathcal{H}\}_{t>0}$ are positive, they have tensor extensions to $L^p(\mathbb{R}^n,\mathbb{B})$
satisfying the same $L^p$-boundedness properties.

If $\ell=1,2$ and $f\in L^p(\mathbb{R}^n,\mathbb{B})$, $1<p<\infty$, we define
$$\mathcal{G}_{\mathcal{L},\mathbb{B}}^\ell (f)(x,t)
    =t^\ell \partial_t^\ell W_t ^\mathcal{L}(f)(x),\quad x\in \mathbb{R}^n, \ t>0, \ n\geq 3,$$
and
$$\mathcal{G}_{\mathcal{H},\mathbb{B}}^\ell(f)(x,t)
    =t^\ell \partial_t^\ell W_t ^\mathcal{H}(f)(x),\quad x\in \mathbb{R}^n, \ t>0, \ n\geq 1.$$

Let $\alpha >-1/2$. The Laguerre operator $\mathcal{L}_\alpha$ is defined by
$$\mathcal{L}_\alpha
    =\frac{1}{2}\left(-\frac{d^2}{dx^2}+x^2+\frac{\alpha ^2-1/4}{x^2}\right), \quad x\in (0,\infty ).$$

If $k\in \mathbb{N}$ we consider the $k$-th Laguerre function
$$\varphi _k^\alpha (x)
    =\Big(\frac{2\Gamma (k+1)}{\Gamma (k+\alpha +1)}\Big)^{1/2}e^{-x^2/2}x^{\alpha +1/2}L_k^\alpha (x^2), \quad x\in (0,\infty ),$$
where $L_k^\alpha$ represents the $k$-th Laguerre polynomial (\cite[pp. 100--102]{Sz}).
The family $\{\varphi _k^\alpha \}_{k\in \mathbb{N}}$
is orthonormal and complete in $L^2(0,\infty)$.
Moreover, for every $k\in \mathbb{N}$,
$$\mathcal{L}_\alpha \varphi _k^\alpha
    =(2k+\alpha +1)\varphi _k^\alpha .$$
The semigroup of operators $\{W_t^{\mathcal{L}_\alpha}\}_{t>0}$ generated by $-\mathcal{L}_\alpha$ in $L^2(0,\infty )$ is defined by
$$W_t^{\mathcal{L}_\alpha }(f)
    =\sum_{k=0}^\infty e^{-t(2k+\alpha +1)}c_k^\alpha (f)\varphi _k^\alpha ,\quad t>0 \mbox{ and }f  \in L^2(0,\infty ),$$
where $c_k^\alpha (f)=\int_0^\infty \varphi _k^\alpha (y)f(y)dy$, $k\in \mathbb{N}$.

\noindent According to the Mehler's formula (\cite[(1.1.47)]{Th}) we can write, for every $t>0$,
\begin{equation}\label{WLaguerre}
    W_t^{\mathcal{L}_\alpha }(f)(x)
        =\int_0^\infty W_t^\alpha (x,y)f(y)dy,\quad f\in L^2(0,\infty),
\end{equation}
where, for each $x,y,t\in (0,\infty)$
$$W_t^\alpha (x,y)
    =\left(\frac{2e^{-t}}{1-e^{-2t}}\right)^{1/2}\left(\frac{2xye^{-t}}{1-e^{-2t}}\right)^{1/2}I_\alpha \left(\frac{2xye^{-t}}{1-e^{-2t}}\right)
    \exp\left[-\frac{1}{2}(x^2+y^2)\frac{1+e^{-2t}}{1-e^{-2t}}\right],$$
and $I_\alpha$ denotes the modified Bessel function of the first kind and order $\alpha$.

If we define, for every $t>0$, $W_t^{\mathcal{L}_\alpha}$ on $L^p(0,\infty )$, $1\leq p\leq \infty$ by \eqref{WLaguerre},
then $\{W_t^{\mathcal{L}_\alpha}\}_{t>0}$ is a positive semigroup of contractions in $L^p(0,\infty )$, $1\leq p\leq \infty$.
Moreover, for every $t>0$, $W_t^{\mathcal{L}_\alpha }$ can be extended to $L^p((0,\infty ),\mathbb{B})$ preserving the $L^p$-boundedness properties.

If $\ell=1,2$ we consider
$$\mathcal{G}_{\mathcal{L}_\alpha , \mathbb{B}}^\ell (f)(x,t)
    =t^\ell \partial_t^\ell W_t^{\mathcal{L}_\alpha}(f)(x), \quad x,t\in (0,\infty),$$
for every $f\in L^p((0,\infty ),\mathbb{B})$, $1<p<\infty$.

We now establish the main result of this paper.

\begin{Th}\label{mean}
    Let $\mathbb{B}$ be a Banach space and $\alpha>-1/2$. The following assertions are equivalent.

    $(a)$ $\mathbb{B}$ is UMD.

    $(b)$ For $\ell =1,2$ and for every (equivalently, for some) $1<p<\infty$,
    $$\|\mathcal{G}_{\mathcal{H},\mathbb{B}}^\ell (f)\|_{L^p(\mathbb{R}^n,\gamma (H,\mathbb{B}))}
        \sim \|f\|_{L^p(\mathbb{R}^n,\mathbb{B})},\quad f\in L^p(\mathbb{R}^n,\mathbb{B}),\;n\geq 1.$$

    $(c)$ For $\ell =1,2$ and for every (equivalently, for some) $1<p<\infty$,
    $$\|\mathcal{G}_{\mathcal{L},\mathbb{B}}^\ell (f)\|_{L^p(\mathbb{R}^n,\gamma (H,\mathbb{B}))}
        \sim \|f\|_{L^p(\mathbb{R}^n,\mathbb{B})},\quad f\in L^p(\mathbb{R}^n,\mathbb{B}),\;n\geq 3.$$

    $(d)$ For $\ell =1,2$ and for every (equivalently, for some) $1<p<\infty$,
    $$\|\mathcal{G}_{\mathcal{L}_\alpha,\mathbb{B}}^\ell(f)\|_{L^p((0,\infty),\gamma (H,\mathbb{B}))}
        \sim \|f\|_{L^p((0,\infty ),\mathbb{B})},\quad f\in L^p((0,\infty),\mathbb{B}).$$
\end{Th}
Note that, since $\gamma (H,\mathbb{C})=H$, the equivalences in Theorem~\ref{mean}, ({\it b}), ({\it c}) and ({\it d})
are Banach valued versions of the corresponding scalar equivalences
(see  \cite{BMR}, \cite{Tang}, \cite[Chapter 4]{Th} and \cite{Wrob}).

In \cite{BCCFR} we study square functions associated to the subordinated Poisson semigroup for the Hermite operator in a Banach valued setting.
By using auxiliar operators and Cauchy-Riemann type equations adapted to the Hermite setting we characterized the UMD Banach spaces.
We remark that, as it can be observed in \cite{Hy}, \cite{MTX} and \cite{Xu}, in order to describe geometric properties of Banach spaces
(UMD, $q$-martingale type and cotype,...) by using square functions, subordinated (Poisson) diffusion semigroups must be considered.
Moreover, in \cite{Hy}, Hyt\"onen dealt with diffusion semigroups and the semigroups $\{W_t^\mathcal{H}\}_{t>0}$, $\{W_t^\mathcal{L}\}_{t>0}$
 and $\{W_t^{\mathcal{L}_\alpha }\}_{t>0}$ are not diffusion semigroups because they are not conservative.
Then, in particular the results in \cite{BCCFR} are not covered by the ones in \cite{Hy}. The results obtained by Hyt\"onen for general diffusion semigroups
in a UMD setting are weaker than the ones got for subordinated diffusion semigroups (\cite[Theorem 5.1]{Hy}). In order to get a better result for every
diffusion semigroups Hyt\"onen reduced the admisible class of Banach spaces. He considered the class of Banach spaces which are isomorphic to a closed
subspace of a  complex interpolation space $[Z,Y]_\theta$ where $Z$ is a Hilbert space, $Y$ is a UMD Banach space and $0<\theta <1$. We write $\zeta$
to refer this class of Banach spaces. $\zeta$ contains all the standard UMD spaces.  In \cite{Rub} Rubio de Francia posed the question whether the
equality $\zeta = \ $UMD holds. As far as we know this question remains open.

In contrast with the results in \cite{Hy}  we get Theorem~\ref{mean} for the semigroups
$\{W_t^\mathcal{H}\}_{t>0}$, $\{W_t^\mathcal{L}\}_{t>0}$ and $\{W_t^{\mathcal{L}_\alpha }\}_{t>0}$ which are not
diffusion semigroups and, as it was above mentioned, they are not conservative. In order to prove Theorem~\ref{mean}
we use a procedure different to the one  used in \cite{Hy}. For establishing that if $\mathbb{B}$ is a UMD Banach space the  equivalences
in ({\it b}), ({\it c}) and ({\it d}) hold, we take advantage of the following fact: close to singularities, our operators are good perturbations
of the corresponding operators associated with the Laplacian operator. The exact meaning of this idea is clear in the proof. Then, we use
\cite[Theorem 4.2]{KW1}. To see that the equivalences in ({\it b}), ({\it c}) and ({\it d}) imply that $\mathbb{B}$
is UMD, we have taken into account that the UMD Banach spaces are characterized by the $L^p$-boundedness properties of the imaginary powers
$\mathcal{H}^{i\gamma}$, $\mathcal{L}^{i\gamma }$,  $\mathcal{L}_\alpha ^{i\gamma}$, $\gamma >0$ of $\mathcal{H}$, $\mathcal{L}$  and
$\mathcal{L}_\alpha $, respectively (\cite[Theorem 1.2]{BCCR} and \cite[Theorem 3]{BCFR}).

In the next sections we prove our result for the Hermite operator in $\mathbb{R}^n$, $n\geq 1$ (Section~\ref{sec:Hermite}),
the Schr\"odinger operators in $\mathbb{R}^n$, $n\geq 3$ (Section~\ref{sec:Schorindger}) and
the Laguerre operators in $(0,\infty )$ (Section~\ref{sec:Laguerre}).

Throughout this paper by $C$ and $c$ we always denote positive constants that can change in each occurrence.\\

\noindent {\it \textbf{Acknowledgements}}. The authors wish to thank Professor Peter Sj\"ogren for posing us, after knowing our
results in \cite{BCCFR}, the question of dealing with the heat semigroup for the Hermite operator.

\section{Proof of Theorem~\ref{mean} for the Hermite operator}\label{sec:Hermite}

In this section we prove $(a) \Leftrightarrow (b)$ in Theorem~\ref{mean}.

\subsection{} \fbox{$(a) \Rightarrow (b)$} Let $\ell =1,2$, $n \geq 1$ and $1<p<\infty$. We define $\mathcal{G}^\ell_{-\Delta ,\B}(f)$, for
every $f \in L^p(\R,\B)$, as follows
$$\mathcal{G}^\ell_{-\Delta ,\B}(f)(x,t)
    =t^\ell \partial_t^\ell W_t(f)(x), \quad x \in \R \mbox{ and } t>0.$$
Assume that $\B$ is a UMD Banach space.

We start proving that
$$\|\mathcal{G}_{\mathcal{H},\B}^\ell (f)\|_{L^p(\mathbb{R}^n,\gamma (H,\mathbb{B}))}
    \leq C\|f\|_{L^p(\mathbb{R}^n,\mathbb{B})},\quad f\in L^p(\mathbb{R}^n,\mathbb{B}).$$

Let $f\in L^p(\mathbb{R}^n,\mathbb{B})$. We can write
\begin{equation}\label{2.1}
    \partial_t^\ell W_t^\mathcal{H}(f)(x)
        =\int_{\mathbb{R}^n} \partial_t^\ell W_t^\mathcal{H}(x,y)f(y)dy,\quad x\in \mathbb{R}^n\mbox{ and }t>0.
\end{equation}
Derivation under the integral sign is justified. Indeed, we have, for every $x,y\in \mathbb{R}^n\mbox{ and }t>0$,
\begin{align}\label{2.2}
    \partial_t W_t^\mathcal{H}(x,y)
        =& \frac{1}{\pi ^{n/2}}\left(\frac{e^{-2t}}{1-e^{-4t}}\right)^{n/2}\exp\left[-\frac{1}{4}\left(|x-y|^2\frac{1+e^{-2t}}{1-e^{-2t}}+|x+y|^2\frac{1-e^{-2t}}{1+e^{-2t}}\right)\right]\nonumber\\
        &\times \left[-n\frac{1+e^{-4t}}{1-e^{-4t}}+|x-y|^2\frac{e^{-2t}}{(1-e^{-2t})^2}
                -|x+y|^2\frac{e^{-2t}}{(1+e^{-2t})^2}\right],
\end{align}
and
\begin{align}\label{2.3}
    \partial_t^2 W_t^\mathcal{H}(x,y)
        =& \frac{1}{\pi ^{n/2}}\left(\frac{e^{-2t}}{1-e^{-4t}}\right)^{n/2}\exp\left[-\frac{1}{4}\left(|x-y|^2\frac{1+e^{-2t}}{1-e^{-2t}}+|x+y|^2\frac{1-e^{-2t}}{1+e^{-2t}}\right)\right]\nonumber\\
        &\times\left\{\left[-n\frac{1+e^{-4t}}{1-e^{-4t}}+|x-y|^2\frac{e^{-2t}}{(1-e^{-2t})^2}
                -|x+y|^2\frac{e^{-2t}}{(1+e^{-2t})^2}\right]^2\right.\nonumber\\
        &+\left.\frac{8ne^{-4t}}{(1-e^{-4t})^2}-|x-y|^2\frac{2e^{-2t}(1+e^{-2t})}{(1-e^{-2t})^3}
                +|x+y|^2\frac{2e^{-2t}(1-e^{-2t})}{(1+e^{-2t})^3}\right\}.
\end{align}
Hence, we deduce that, for $k=0,1,2$,
\begin{equation}\label{2.4}
    \left|t^k \partial_t^k W_t^\mathcal{H}(x,y)\right|
        \leq C\frac{t^k e^{-nt}e^{-c|x-y|^2/t}}{(1-e^{-2t})^{n/2+k}}
        \leq C\frac{e^{-c|x-y|^2/t}}{t^{n/2}},\quad x,y\in\mathbb{R}^n \mbox{ and }t>0.
\end{equation}
Estimation \eqref{2.4} justifies the derivation under the integral sign in \eqref{2.1}.

We split the operators $\mathcal{G}_{\mathcal{H},\B}^\ell$ and $\mathcal{G}_{-\Delta ,\B}^\ell$ as follows. We write, for
$\mathcal{Q}=\mathcal{H}$ or $\mathcal{Q}=-\Delta$,
$$\mathcal{G}_{\mathcal{Q},\B}^\ell
    = \mathcal{G}_{\mathcal{Q},\B,{\rm loc}}^\ell + \mathcal{G}_{\mathcal{Q},\B,{\rm glob}}^\ell,$$
where
\begin{equation}\label{7.1}
    \mathcal{G}_{\mathcal{Q},\B,{\rm loc}}^\ell(f)(x,t)
        =\mathcal{G}_{\mathcal{Q},\B}^\ell(\chi_{B(x,\rho (x))}(y)f(y))(x,t),\quad x\in \mathbb{R}^n,t>0,
\end{equation}
and
$$ \rho(x)=\left\{
        \begin{array}{ll}
            \dfrac{1}{2}, &|x| \leq 1\\
            \\
            \dfrac{1}{1+|x|}, & |x| > 1
        \end{array}\right. .$$
For every $x \in \R$, $\rho(x)$ is called the critical radius in $x$ (see \cite[p. 516]{Sh}).

We consider the following decomposition of the operator $\mathcal{G}_{\mathcal{H},\B}^\ell$:
$$\mathcal{G}_{\mathcal{H},\B}^\ell
    =\sum_{j=1}^3T_{j,\B}^\ell,$$
where
$T_{1,\B}^\ell=\mathcal{G}_{\mathcal{H},\B,{\rm loc}}^\ell-\mathcal{G}_{-\Delta,\B,{\rm loc}}^\ell$,
$T_{2,\B}^\ell=\mathcal{G}_{\mathcal{H},\B,{\rm glob}}^\ell$ and
$T_{3,\B}^\ell=\mathcal{G}_{-\Delta ,\B,{\rm loc}}^\ell$.

\begin{Lem}\label{Lem T}
    Let $\B$ be a UMD Banach space and $j=1,2,3$. Then, there exits $C>0$ verifying that
    \begin{equation}\label{2.5}
        \|T_{j,\B}^\ell (f)\|_{L^p(\mathbb{R}^n,\gamma (H,\mathbb{B}))}
            \leq C\|f\|_{L^p(\mathbb{R}^n,\mathbb{B})}, \quad f \in L^p(\mathbb{R}^n,\mathbb{B}).
    \end{equation}
\end{Lem}

\begin{proof}[Proof of Lemma~\ref{Lem T} for $T_{3,\B}^\ell$]
    We consider $\varphi^\ell (x)=\Big( \partial_t^\ell G_{\sqrt{t}}(x) \Big)_{|t=1}$, $x\in \mathbb{R}^n$.
    Thus, $\varphi^\ell \in S(\mathbb{R}^n)\subset L^2(\mathbb{R}^n)$, where
    $S(\mathbb{R}^n)$ denotes the Schwartz class. Moreover, $\varphi^\ell$  satisfies conditions
    $(C1)$ and $(C2)$ in \cite[p. 111]{KW1}. Indeed, according to
   \cite[p. 121, (23)]{EMOT} we have that
    \begin{align*}
        \widehat{\varphi ^\ell}(y)
            &=\int_{\mathbb{R}^n}e^{-ix\cdot y}\Big[ \partial_t^\ell \left(\frac{e^{-|x|^2/(4t)}}{(4\pi t)^{n/2}}\right)\Big]_{\big|t=1}dx\\
            &= \partial_t^\ell\left[\int_{\mathbb{R}^n}e^{-ix\cdot y}\frac{e^{-|x|^2/(4t)}}{(4\pi t)^{n/2}}dx\right]_{\big|t=1}\\
            &=\partial_t^\ell\left(e^{-t|y|^2}\right)_{\big|t=1}
            =(-|y|^2)^\ell e^{-|y|^2},\quad y\in \mathbb{R}^n.
    \end{align*}
    Now, straightforward manipulations allow us to see that the conditions $(C1)$ and $(C2)$ in \cite[p. 111]{KW1} are satisfied by
    $\varphi^\ell$. On the other hand,
    $\varphi_t^\ell(x)=t^{-n}\varphi ^\ell (x/t)=\Big(s^\ell \partial_s^\ell G_{\sqrt{s}}(x)\Big)_{|s=t^2}$, $x\in \mathbb{R}^n$, and $t>0$.
    Note that if $\{h_n\}_{n \in \mathbb{N}}$ is an orthonormal basis in $H$, then $\{h_n(\sqrt{t})/\sqrt{2}\}_{n \in \mathbb{N}}$ is
    also an orthonormal basis in $H$. Hence,
    $$\| \mathcal{G}_{-\Delta, \B}^\ell (g) (x,\cdot)\|_{\gamma(H,\B)}
        = \sqrt{2} \| (\varphi_\cdot^\ell * g)(x)\|_{\gamma(H,\B)}, \quad g \in S(\R,\B) \text{ and } x \in \R.$$

    Hence, by invoking \cite[Theorem 4.2]{KW1} there exists a bounded operator $\widetilde{\mathcal{G}}_{-\Delta , \B }^\ell$ from
    $L^p(\mathbb{R}^n,\mathbb{B})$ into $L^p(\mathbb{R}^n,\gamma(H,\B))$ such that
$$
\widetilde{\mathcal{G}}_{-\Delta , \B}^\ell(g)=\mathcal{G}_{-\Delta ,\B}^\ell (g), \quad g \in S(\R,\B).
$$
    Let $f\in L^p(\mathbb{R}^n,\mathbb{B})$. We are going to see that $\widetilde{\mathcal{G}}_{-\Delta ,\B}^\ell(f)=\mathcal{G}_{-\Delta ,\B}^\ell (f).$ In order to do
    this we choose a sequence $(f_m)_{m=1}^\infty \subset C_c^\infty(\mathbb{R}^n)\otimes \mathbb{B}$ such that $f_m\longrightarrow f$,
    as $m\rightarrow \infty$, in $L^p(\mathbb{R}^n,\mathbb{B})$. Note that $C_c^\infty (\mathbb{R}^n)\otimes \mathbb{B}\subset S(\mathbb{R}^n,\mathbb{B})$
    is a dense subset of $L^p(\mathbb{R}^n,\mathbb{B})$. It can be shown that
    \begin{equation}\label{8.1}
        \left|t^\ell \partial_t^\ell G_{\sqrt{t}}(x-y)\right|
            \leq C\frac{e^{-c|x-y|^2/t}}{t^{n/2}},\quad x,y \in \mathbb{R}^n\mbox{ and }t>0.
    \end{equation}
    Then, for every $N \in \mathbb{N}$ and $x\in \mathbb{R}^n$, there exists $C_N>0$ for which
    \begin{align*}
        & \|\mathcal{G}_{-\Delta ,\B}^\ell (f)(x,\cdot)-\mathcal{G}_{-\Delta ,\B}^\ell (f_m)(x,\cdot)\|_{L^2\left((1/N,N),\frac{dt}{t};\mathbb{B}\right)}&&\\
        & \qquad \qquad \leq \int_{\mathbb{R}^n}\|f(y)-f_m(y)\|_\mathbb{B}\left\|t^\ell \partial_t^\ell G_{\sqrt{t}}(x-y)\right\|_{L^2\left((1/N,N),\frac{dt}{t}\right)}dy\\
        & \qquad \qquad \leq C_N\int_{\mathbb{R}^n}\|f(y)-f_m(y)\|_\mathbb{B}\left(\int_{1/N}^N \frac{1}{(t+|x-y|^2)^{n+1}}dt\right)^{1/2}dy\\
        & \qquad \qquad \leq C_N\int_{\mathbb{R}^n}\|f(y)-f_m(y)\|_\mathbb{B}\frac{1}{(1/N+|x-y|^2)^{n/2}}dy\\
        & \qquad \qquad \leq C_N\|f-f_m\|_{L^p(\mathbb{R}^n,\mathbb{B})}\left(\int_{\mathbb{R}^n}\frac{1}{(1/\sqrt{N}+|x-y|)^{np'}}dy\right)^{1/p'}\\
        & \qquad \qquad \leq C_N\|f-f_m\|_{L^p(\mathbb{R}^n,\mathbb{B})},\quad m\in \mathbb{N}.
    \end{align*}
    Hence, for every $N\in \mathbb{N}$ and $x\in \mathbb{R}^n$,
    $$\mathcal{G}_{-\Delta ,\B}^\ell (f_m)(x,\cdot)
        \longrightarrow \mathcal{G}_{-\Delta ,\B}^\ell(f)  (x,\cdot), \quad \text{as } m \rightarrow \infty \text{ in } L^2\left((1/N,N),\frac{dt}{t};\mathbb{B}\right).$$

    On the other hand,
    $$\mathcal{G}_{-\Delta ,\B}^\ell (f_m)
        \longrightarrow\widetilde{\mathcal{G}}_{-\Delta , \B}^\ell(f) , \quad \text{as } m\rightarrow \infty \text{ in } L^p(\mathbb{R}^n,\gamma (H,\mathbb{B})).$$
    Then, there exists a subsequence of $(f_m)_{m=1}^\infty$ which we continue denoting by  $(f_m)_{m=1}^\infty$, satisfying
    $$\mathcal{G}_{-\Delta ,\B}^\ell (f_m)(x,\cdot)
        \longrightarrow \widetilde{\mathcal{G}}_{-\Delta , \B}^\ell(f)(x), \quad \text{as } m\rightarrow \infty \text{ in } \gamma (H,\mathbb{B}),$$
    for every $x\in {\bf N}$, where ${\bf N} \subset \R$ and $|\R \setminus {\bf N}|=0$.
    Since $\gamma (H,\mathbb{B})$ is
    continuously contained in $L(H,\mathbb{B})$, we have that, for every $x \in {\bf N}$,
    $$\mathcal{G}_{-\Delta ,\B}^\ell (f_m)(x,\cdot)
        \longrightarrow \widetilde{\mathcal{G}}_{-\Delta , \B}^\ell(f)(x), \quad \text{as } m\rightarrow \infty \text{ in } L(H,\mathbb{B}).$$

    Let $x\in {\bf N}$. We choose $h\in H$ such that its support is compact and contained in $(0,\infty )$. For every $S\in \mathbb{B}^*$ we can write
    \begin{align*}
        & \langle S,[\widetilde{\mathcal{G}}_{-\Delta , \B}^\ell(f)(x)](h)\rangle _{\mathbb{B}^*,\mathbb{B}}
            = \lim_{m \to \infty }\langle S,[\mathcal{G}_{-\Delta ,\B}^\ell (f_m)(x, \cdot)](h)\rangle _{\mathbb{B}^*,\mathbb{B}} \\
        & \qquad  =  \langle S,\int_0^\infty\mathcal{G}_{-\Delta ,\B}^\ell (f)(x,t)h(t)\frac{dt}{t}\rangle _{\mathbb{B}^*,\mathbb{B}}
            = \int_0^\infty \langle S,\mathcal{G}_{-\Delta ,\B}^\ell (f)(x,t)\rangle _{\mathbb{B}^*,\mathbb{B}}h(t)\frac{dt}{t}.
    \end{align*}
    Moreover,
    \begin{align*}
        \left|\int_0^\infty \langle S,\mathcal{G}_{-\Delta ,\B}^\ell (f)(x,t)\rangle _{\mathbb{B}^*,\mathbb{B}}h(t)\frac{dt}{t}\right|
            & = \left|\langle S,[\widetilde{\mathcal{G}}_{-\Delta , \B}^\ell(f)(x)](h)\rangle _{\mathbb{B}^*,\mathbb{B}}\right| \\
            & \leq \|S\|_{\mathbb{B}^*}\|\widetilde{\mathcal{G}}_{-\Delta , \B}^\ell(f)(x)\|_{L(H,\mathbb{B})}\|h\|_\mathbb{B}.
    \end{align*}

    We conclude that $\langle S,\mathcal{G}_{-\Delta ,\B}^\ell(f)(x,\cdot )\rangle _{\mathbb{B}^*,\mathbb{B}}\in H$ and
    $$\langle S,[\widetilde{\mathcal{G}}_{-\Delta , \B}^\ell(f)(x)](w)\rangle _{\mathbb{B}^*,\mathbb{B}}
        =\int_0^\infty \langle S,\mathcal{G}_{-\Delta ,\B}^\ell (f)(x,t)\rangle _{\mathbb{B}^*,\mathbb{B}}w(t)\frac{dt}{t},\quad w\in H.$$
    Thus we prove that $\widetilde{\mathcal{G}}_{-\Delta ,\B}^\ell(f)(x)=\mathcal{G}_{-\Delta ,\B}^\ell (f)(x,\cdot)$ as
    elements of $\gamma (H,\mathbb{B})$.

    We now use the ideas developed in \cite[Proposition 2.3]{HTV} to see that \eqref{2.5} holds for $j=2$.
   According to \cite[Proposition 5]{DGMTZ}, for every $M>0$ there exists $C>0$ such that
    \begin{equation}\label{2.6}
        \frac{1}{C} \leq \frac{\rho(x)}{\rho(y)} \leq C, \quad x \in B(y,M \rho(y)).
    \end{equation}
    We can find a sequence $(x_k)_{k=1}^\infty$ such that
    \begin{itemize}
        \item[$(i)$] $\displaystyle \bigcup_{k=1}^\infty B(x_k,\rho(x_k))=\R$,
        \item[$(ii)$] For every $M>0$ there exists $m \in \mathbb{N}$ such that, for each $j \in \mathbb{N}$, 
        $$\mbox{card }\{ k  \in \mathbb{N} : B(x_k,M \rho(x_k)) \cap B(x_j, M \rho(x_j))\not=\emptyset \} \leq m.$$
    \end{itemize}

    Let $k \in \mathbb{N}$. If $x \in B(x_k,\rho(x_k))$, then \eqref{2.6} implies that $|y-x_k| \leq \rho(x)+\rho(x_k) \leq C_0
    \rho(x_k)$, provided that $y \in B(x,\rho(x))$. Here $C_0>0$ does not depend on $k \in \mathbb{N}$. We can write for every $x \in B(x_k,\rho(x_k))$ and $t>0,$
    \begin{align*}
        \mathcal{G}_{-\Delta ,\B,\rm{loc}}^\ell (f)(x,t)
            & = \mathcal{G}_{-\Delta ,\B}^\ell \left(\chi_{B(x_k,C_0\rho(x_k))}f\right)(x,t)
                + \mathcal{G}_{-\Delta ,\B}^\ell \left(\left( \chi_{B(x,\rho(x))} - \chi_{B(x_k,C_0\rho(x_k))}\right)f\right)(x,t) \\
            & = \mathcal{G}_{-\Delta ,\B}^\ell \left(\chi_{B(x_k,C_0\rho(x_k))}f\right)(x,t)
                - \mathcal{G}_{-\Delta ,\B}^\ell \left(\chi_{B(x_k,C_0\rho(x_k)) \setminus B(x,\rho(x))}f\right)(x,t).
    \end{align*}

    Let $x \in B(x_k,\rho(x_k))$. We consider the operator
    $$L_{k,x}(f)(t)
        = \mathcal{G}_{-\Delta ,\B}^\ell \left(\chi_{B(x_k,C_0\rho(x_k)) \setminus B(x,\rho(x))}f\right)(x,t), \quad t>0. $$
    By \eqref{8.1} we have that
$$
        \left\| t^\ell \partial_t^\ell G_{\sqrt{t}}(x-y) \right\|_H
        \leq C \left( \int_0^\infty \frac{e^{-c|x-y|^2/t}}{t^{n+1}} dt \right)^{1/2}
            \leq \frac{C}{|x-y|^n}, \quad y \in \R \setminus\{x\}.
$$
    Hence, for every $y \notin B(x,\rho(x))$, the function $g_{x,y}(t)=t^\ell \partial_t^\ell G_{\sqrt{t}}(x-y)$, $t \in (0,\infty)$, belongs to $H$
    and $\|g_{x,y}\|_H \leq C/\rho(x)^n$. Then, $L_{k,x}(f) \in L^2((0,\infty),dt/t;\B)$ and
    \begin{align*}
        \|L_{k,x}(f)\|_{L^2\left((0,\infty),\frac{dt}{t};\B\right)}
            & \leq \int_{B(x_k,C_0\rho(x_k)) \setminus B(x,\rho(x))} \|g_{x,y}\|_H \|f(y)\|_\B dy \\
            & \leq \frac{C}{\rho(x)^n} \int_{B(x_k,C_0 \rho(x_k))} \|f(y)\|_\B dy.
    \end{align*}
    Hence, $L_{k,x}(f) \in \gamma(H,\B)$. Indeed, by (\ref{normsup})  we have
    $$\|L_{k,x}(f)\|_{\gamma(H,\B)}
        = \sup \left( \mathbb{E} \left\| \sum_j \gamma_j \int_0^\infty L_{k,x}(f)(t) h_j(t) \frac{dt}{t} \right\|_\B^2 \right)^{1/2},$$
  where $(\gamma_j)_{j=1}^\infty$ is a sequence 
    of independent standard Gaussian random variables and the supremum is taken over all the finite families $\{h_j\}$ of orthonormal functions in $H$. Suppose that $(h_j)_{j=1}^m$ is
    an orthonormal set in $H$.
    We can write
    \begin{align*}
        & \left( \mathbb{E} \left\| \sum_{j=1}^m \gamma_j \int_0^\infty L_{k,x}(f)(t) h_j(t) \frac{dt}{t} \right\|_\B^2 \right)^{1/2} \\
        & \qquad \qquad = \left( \mathbb{E} \left\| \sum_{j=1}^m \gamma_j
                            \int_0^\infty \int_{B(x_k,C_0\rho(x_k)) \setminus B(x,\rho(x))} g_{x,y}(t) f(y) dy h_j(t) \frac{dt}{t} \right\|_\B^2 \right)^{1/2} \\
        & \qquad \qquad = \left( \mathbb{E} \left\|  \int_{B(x_k,C_0\rho(x_k)) \setminus B(x,\rho(x))} f(y)
                            \sum_{j=1}^m \gamma_j \int_0^\infty g_{x,y}(t)   h_j(t) \frac{dt}{t} dy\right\|_\B^2 \right)^{1/2} \\
        & \qquad \qquad \leq  \int_{B(x_k,C_0\rho(x_k)) \setminus B(x,\rho(x))} \|f(y)\|_\B
                            \left( \mathbb{E} \left| \sum_{j=1}^m \gamma_j \int_0^\infty g_{x,y}(t)   h_j(t) \frac{dt}{t} \right|^2 \right)^{1/2} dy\\
        & \qquad \qquad \leq  \int_{B(x_k,C_0\rho(x_k)) \setminus B(x,\rho(x))} \|f(y)\|_\B \|g_{x,y}\|_{\gamma(H,\mathbb{C})} dy\\
        & \qquad \qquad =  \int_{B(x_k,C_0\rho(x_k)) \setminus B(x,\rho(x))} \|f(y)\|_\B \|g_{x,y}\|_{H} dy\\
        & \qquad \qquad \leq \frac{C}{\rho(x)^n}  \int_{B(x_k,C_0\rho(x_k))} \|f(y)\|_\B dy.
    \end{align*}
    Hence,
    $$\|L_{k,x}(f)\|_{\gamma(H,\B)}
        \leq \frac{C}{\rho(x)^n}  \int_{B(x_k,C_0\rho(x_k))} \|f(y)\|_\B dy.$$
    By using \eqref{2.6}, we deduce that $B(x_k,C_0\rho(x_k)) \subset B(x,C_1\rho(x))$, where $C_1$ does not depend on $k$ neither on $x$. Then, we get
    $$\|L_{k,x}(f)\|_{\gamma(H,\B)}
        \leq \frac{C}{\rho(x)^n}  \int_{B(x,C_1\rho(x))} \|f(y)\|_\B dy
        \leq \mathcal{M}(\|f\|_\B)(x),$$
    where $\mathcal{M}$ denotes the Hardy-Littlewood maximal function.

    According to the classical maximal theorem and the boundedness of $\mathcal{G}_{-\Delta ,\B}^\ell$
    from $L^p(\R,\B)$ into $L^p(\R,\gamma(H,\B))$, we obtain
    \begin{align*}
        \| \mathcal{G}_{-\Delta ,\B,\rm{loc}}^\ell (f) \|^p_{L^p(\R,\gamma(H,\B))}
            \leq & \sum_{k=1}^\infty \int_{B(x_k,\rho(x_k))} \| \mathcal{G}_{-\Delta ,\B,\rm{loc}}^\ell (f) \|^p_{\gamma(H,\B)} dx \\
            \leq & C \sum_{k=1}^\infty \Big(
                    \int_{\R} \| \mathcal{G}_{-\Delta ,\B}^\ell \left(\chi_{B(x_k, C_0\rho(x_k))}  f \right)(x,\cdot) \|^p_{\gamma(H,\B)} dx \\
            & + \int_{\R} \| \mathcal{G}_{-\Delta ,\B}^\ell \left(\chi_{B(x_k,C_0\rho(x_k)) \setminus B(x,\rho(x))}  f \right)(x,\cdot) \|^p_{\gamma(H,\B)} dx \Big)\\
             \leq & C  \Big( \sum_{k=1}^\infty \int_{B(x_k, C_0\rho(x_k))} \|f(y)\|_\B^p dy
                    + \int_{\R} \left| \mathcal{M}(\|f\|_\B)(x)  \right|^p dx \Big) \\
            \leq &C \int_{\R} \|f(y)\|_\B^p dy.
    \end{align*}
    We conclude that \eqref{2.5} holds for $T_{3,\B}^\ell$.
\end{proof}

\begin{proof}[Proof of Lemma~\ref{Lem T} for $T_{1,\B}^\ell$]
    By using the perturbation formula (\cite[(5.25)]{DGMTZ}) we can write
    \begin{align*}
        \partial_t\left[G_{\sqrt{t}}(x-y)-W_t^\mathcal{H}(x,y)\right]
            = & \int_0^{t/2}\int_{\mathbb{R}^n}|z|^2 \left[\partial_u G_{\sqrt{u}}(x-z)\right]_{\big|u=t-s}W_s^\mathcal{H}(z,y)dzds\\
            &+\int_0^{t/2}\int_{\mathbb{R}^n}|z|^2 G_{\sqrt{s}}(x-z) \left[\partial_u W_u^\mathcal{H}(z,y)\right]_{\big|u=t-s}dzds\\
            &+\int_{\mathbb{R}^n}|z|^2 G_{\sqrt{t/2}}(x-z)W_{t/2}^\mathcal{H}(z,y)dz\\
            = &H_1^1(x,y,t)+H_2^1(x,y,t)+H_3^1(x,y,t),\quad x,y\in \mathbb{R}^n\mbox{ and }t>0.
    \end{align*}
    Then, it follows that
    \begin{align*}
        & \partial_t^2 \left[G_{\sqrt{t}}(x-y)-W_t^\mathcal{H}(x,y)\right]
            =\int_0^{t/2}\int_{\mathbb{R}^n}|z|^2\Big[ \partial_u^2 G_{\sqrt{u}}(x-z)\Big]_{\big|u=t-s}W_s^\mathcal{H}(z,y)dzds\\
        & \qquad \qquad + \int_0^{t/2}\int_{\mathbb{R}^n}|z|^2 G_{\sqrt{s}}(x-z)\Big[ \partial_u^2W_u^\mathcal{H}(z,y)\Big]_{\big|u=t-s}dzds\\
        & \qquad \qquad+\int_{\mathbb{R}^n}|z|^2\left(\Big[ \partial_u G_{\sqrt{u}}(x-z)\Big]_{\big|u=t/2}W_{t/2}^\mathcal{H}(z,y)+
                                    G_{\sqrt{t/2}}(x-z)\Big[ \partial_u W_u^\mathcal{H}(z,y)\Big]_{\big|u=t/2}\right)dz\\
        & \qquad = H_1^2(x,y,t)+H_2^2(x,y,t)+H_3^2(x,y,t),\quad x,y\in \mathbb{R}^n\mbox{ and }t>0.
    \end{align*}
    The following estimation will be very useful in the sequel. For every $x\in \mathbb{R}^n$ and $t>0$, we get
    $$\int_{\mathbb{R}^n}e^{-c|x-y|^2/t}|y|^2dy
        \leq C\int_{\mathbb{R}^n}e^{-c|z|^2/t}(|z|^2+|x|^2)dz\leq Ct^{n/2}(t+|x|^2).$$
    Then, we obtain
    \begin{equation}\label{2.7}
        \int_{\mathbb{R}^n}e^{-c|x-y|^2/t}|y|^2dy
            \leq C\frac{t^{n/2}}{\rho (x)^2},\quad x\in \mathbb{R}^n\mbox{ and }0<t\leq \rho (x) ^2.
    \end{equation}

    Minkowski's inequality leads to
    \begin{align*}
        \|T_{1,\B}^\ell(f)(x,\cdot)\|_{L^2\left( (0,\infty),\frac{dt}{t}; \B \right)}
            &\leq \int_{B(x,\rho(x))}\|f(y)\|_\mathbb{B} \left(\int_0^\infty \left|t^\ell \partial_t^\ell [G_{\sqrt{t}}(x-y)-W_t^\mathcal{H}(x,y)]\right|^2\frac{dt}{t}\right)^{1/2}dy\\
            &\leq C\sum_{j=1}^3\int_{B(x,\rho(x))}\|f(y)\|_\mathbb{B}\left(\int_0^\infty |t^\ell H_j^\ell (x,y,t)|^2\frac{dt}{t}\right)^{1/2}dy, \quad x \in \R.
    \end{align*}

    We now study
    $$A_j^\ell (x,y)
        =\left(\int_0^{\rho (x)^2}|t^\ell H_j^\ell (x,y,t)|^2\frac{dt}{t}\right)^{1/2},\quad x,y\in\mathbb{R}^n\mbox{ and }j=1,2,3.$$

    According to \eqref{2.4}, \eqref{8.1}, \eqref{2.6} and \eqref{2.7} we get
    \begin{align*}
        &\left(\int_0^{\rho (x)^2}|t^\ell H_1^\ell (x,y,t)|^2\frac{dt}{t}\right)^{1/2}
            \leq C\left(\int_0^{\rho (x)^2} \left(\int_0^{t/2}\int_{\mathbb{R}^n}|z|^2\frac{e^{-c|x-z|^2/(t-s)}}{(t-s)^{n/2}}\frac{e^{-c|y-z|^2/s}}{s^{n/2}}dzds\right)^2 \frac{dt}{t}\right)^{1/2}\\
        & \qquad \qquad \leq C\left(\int_0^{\rho (x)^2}\frac{e^{-c(|x-y|^2+|y-z|^2)/t}}{t^{n+1}}\left(\int_0^{t/2}\int_{\mathbb{R}^n}|z|^2\frac{e^{-c|y-z|^2/s}}{s^{n/2}}dzds\right)^2dt\right)^{1/2}\\
        & \qquad \qquad \leq \frac{C}{\rho (x)^2}\left(\int_0^{\rho (x)^2}\frac{e^{-c|x-y|^2/t}}{t^{n-1}}dt\right)^{1/2}\leq \frac{C}{\rho (x)^2|x-y|^{n-1/2}}\left(\int_0^{\rho (x)^2}\sqrt{t}dt\right)^{1/2}\\
        & \qquad \qquad \leq \frac{C}{\sqrt{\rho (x)}|x-y|^{n-1/2}},\quad x \in \mathbb{R}^n, \;y\in B(x,\rho (x)), \;x\not=y.
    \end{align*}
    Also by taking into account \eqref{2.4} and again \eqref{2.7}  it follows that
    \begin{align*}
        & \left(\int_0^{\rho (x)^2}|t^\ell H_2^\ell (x,y,t)|^2\frac{dt}{t}\right)^{1/2}
            \leq C\left(\int_0^{\rho (x)^2} \left(\int_0^{t/2}\int_{\mathbb{R}^n}|z|^2\frac{e^{-c|x-z|^2/s}}{s^{n/2}}\frac{e^{-c|y-z|^2/(t-s)}}{(t-s)^{n/2}}dzds\right)^2 \frac{dt}{t}\right)^{1/2}\\
        & \qquad \qquad \leq C\left(\int_0^{\rho (x)^2}\frac{e^{-c(|x-y|^2+|y-z|^2)/t}}{t^{n+1}}\left(\int_0^{t/2}\int_{\mathbb{R}^n}|z|^2\frac{e^{-c|x-z|^2/s}}{s^{n/2}}dzds\right)^2dt\right)^{1/2}\\
        & \qquad \qquad \leq \frac{C}{\sqrt{\rho (x)}|x-y|^{n-1/2}},\quad x,y\in \mathbb{R}^n, x\not=y,
    \end{align*}
    and
    \begin{align*}
        &\left(\int_0^{\rho (x)^2}|t^\ell H_3^\ell (x,y,t)|^2\frac{dt}{t}\right)^{1/2}
            \leq C\left(\int_0^{\rho (x)^2} \left(\int_{\mathbb{R}^n}|z|^2\frac{e^{-c(|x-z|^2+|y-z|^2)/t}}{t^{n -1}}dz\right)^2 \frac{dt}{t}\right)^{1/2}\\
        & \qquad \qquad \leq C\left(\int_0^{\rho (x)^2}\frac{e^{-c|x-y|^2/t}}{t^{2n-1}}\left(\int_{\mathbb{R}^n}|z|^2e^{-c|x-z|^2/t}dz\right)^2dt\right)^{1/2}\\
        & \qquad \qquad \leq \frac{C}{\rho (x)^2}\left(\int_0^{\rho (x)^2}\frac{e^{-c|x-y|^2/t}}{t^{n-1}}dt\right)^{1/2} \leq \frac{C}{\sqrt{\rho (x)}|x-y|^{n-1/2}},\quad x,y\in \mathbb{R}^n, x\not=y.
    \end{align*}
    By combining the above estimations we obtain
    \begin{align}\label{2.8}
        \sum_{j=1}^3\int_{B(x,\rho (x))}\|f(y)\|_\mathbb{B}A_j^\ell (x,y)dy
            & \leq C\int_{B(x,\rho (x))}\frac{\|f(y)\|_\mathbb{B}}{\sqrt{\rho (x)}|x-y|^{n-1/2}}dy\nonumber\\
            & \leq C\sum_{m=0}^\infty \frac{1}{\sqrt{\rho (x)}}\int_{2^{-m-1}\rho (x)\leq |x-y|<2^{-m}\rho (x)} \frac{\|f(y)\|_\mathbb{B}}{|x-y|^{n-1/2}}dy\nonumber\\
            & \leq C\sum_{m=0}^\infty \frac{1}{\rho (x)^ n2^{-m(n-1/2)}}\int_{B(x,2^{-m}\rho (x))} \|f(y)\|_\mathbb{B}dy\nonumber\\
            &\leq C\sum_{m=0}^\infty \frac{1}{2^{m/2}}\mathcal{M}(\|f\|_\mathbb{B})(x),\quad x\in \mathbb{R}^n.
    \end{align}

    On the other hand, \eqref{2.4} and \eqref{8.1} lead to
    \begin{align}\label{2.9}
        & \int_{B(x,\rho (x))}\|f(y)\|_\mathbb{B}\left(\int_{\rho(x)^2}^\infty \left| t^\ell\partial_t^\ell [G_{\sqrt{t}}(x-y)-W_t^\mathcal{H}(x,y)]\right|^2\frac{dt}{t}\right)^{1/2}dy \nonumber\\
        & \qquad \qquad \leq C\int_{B(x,\rho (x))}\|f(y)\|_\mathbb{B}\left(\int_{\rho(x)^2}^\infty \frac{e^{-c|x-y|^2/t}}{t^{n+1}}dt\right)^{1/2}dy\nonumber\\
        & \qquad \qquad \leq C \left(\int_{\rho(x)^2}^\infty \frac{1}{t^{n+1}}dt\right)^{1/2}\int_{B(x,\rho (x))}\|f(y)\|_\mathbb{B}dy
                        \leq \frac{1}{\rho (x)^n}\int_{B(x,\rho (x))}\|f(y)\|_\mathbb{B} dy \nonumber\\
        & \qquad \qquad \leq C\mathcal{M}(\|f\|_\mathbb{B})(x),\quad x\in \mathbb{R}^n.
    \end{align}

    From \eqref{2.8}, \eqref{2.9} we conclude that
    $$\|T_{1,\B}^\ell (f)(x,\cdot)\|_{L^2\left( (0,\infty), \frac{dt}{t}; \B \right)}
        \leq C \mathcal{M}(\|f\|_\B)(x), \quad x \in \R.$$
    Then, $T_{1,\B}^\ell (f)(x,\cdot) \in \gamma(H,\B)$ and by proceeding as above we show that
    $$\|T_{1,\B}^\ell (f)(x,\cdot)\|_{\gamma(H,\B)}
        \leq C \mathcal{M}(\|f\|_\B)(x), \quad x \in \R.$$
    Classical maximal theorems leads to
    $$\|T_{1,\B}^\ell (f)\|_{L^p(\R,\gamma(H,\B))}
        \leq C \|f\|_{L^p(\R,\B)}.$$
\end{proof}

\begin{proof}[Proof of Lemma~\ref{Lem T} for $T_{2,\B}^\ell$]
    By taking into account the following estimation (see \cite[(4.4) and (4.5)]{BCFST})
    $$\exp\left[-c\left(|x-y|^2\frac{1+e^{-2t}}{1-e^{-2t}}+|x+y|^2\frac{1-e^{-2t}}{1+e^{-2t}}\right)\right]
        \leq C\exp\left[-c(|x|+|y|)|x-y|\right],\quad x,y\in \mathbb{R}^n \text{ and } t>0,$$
    we get, by using \eqref{2.2} and \eqref{2.3} ,
    $$\left|t^\ell \partial_t^\ell W_t^\mathcal{H}(x,y)\right|
        \leq C\frac{e^{-c(|x|+|y|)|x-y|}e^{-c|x-y|^2/t}}{t^{n/2}},\quad x,y\in \mathbb{R}^n\mbox{ and }t>0.$$
    Hence, Minkowski's inequality allows us to write
    \begin{align*}
        & \|T_{2,\B}^\ell (f)(x,\cdot)\|_{L^2\left( (0,\infty), \frac{dt}{t}; \B \right)}
            \leq \int_{|x-y|>\rho (x)}\|f(y)\|_\mathbb{B} \left( \int_0^\infty \left|t^\ell \partial_t^\ell W_t^\mathcal{H}(x,y) \right|^2 \frac{dt}{t} \right)^{1/2} dy \\
        & \qquad \qquad \leq C \int_{|x-y|>\rho (x)}\|f(y)\|_\mathbb{B} e^{-c(|x|+|y|)|x-y|} \left( \int_0^\infty \frac{e^{-c|x-y|^2/t}}{t^{n+1}} dt \right)^{1/2} dy \\
        & \qquad \qquad \leq C \int_{|x-y|>\rho (x)}\|f(y)\|_\mathbb{B} \frac{e^{-c(|x|+|y|)|x-y|}}{|x-y|^n} dy \\
        & \qquad \qquad \leq C\sum_{m=0}^\infty  \frac{1}{(2^m\rho (x))^n} \int_{2^m \rho(x) <|x-y| \leq 2^{m+1}\rho (x)}\|f(y)\|_\mathbb{B} e^{-c(|x|+|y|)2^m \rho(x)}dy,\quad x\in \mathbb{R}^n.
    \end{align*}
    Note that if $|x-y|>\rho (x)$, then
    $$(|x|+|y|)\rho (x)\geq |x-y|\rho(x) > \rho(x)^2 = \frac{1}{4}, \quad \text{when } |x| \leq 1,$$
    and
    $$(|x|+|y|)\rho (x)\geq  \frac{|x|}{1+|x|} >  \frac{1}{2}, \quad \text{when } |x| > 1.$$
    Hence,
    \begin{align*}
        \|T_{2,\B}^\ell (f)(x,\cdot)\|_{L^2\left( (0,\infty), \frac{dt}{t}; \B \right)}
            &\leq C\sum_{m=0}^\infty  \frac{e^{-c2^m}}{(2^m\rho (x))^n}\int_{|x-y|\leq 2^{m+1}\rho (x)}\|f(y)\|_\mathbb{B}dy\\
            & \leq C\mathcal{M}(\|f\|_\mathbb{B})(x),\quad x\in \mathbb{R}^n,
    \end{align*}
    and we get that $T_{2,\B}^\ell (f)(x,\cdot) \in \gamma(H,\B)$, $x \in \R$, and
    $$\|T_{2,\B}^\ell (f)(x,\cdot)\|_{\gamma(H,\B)}
        \leq C \mathcal{M}(\|f\|_\B)(x), \quad x \in \R.$$
    Maximal theorem implies now that \eqref{2.5} holds for $T_{2,\B}^\ell$.
\end{proof}

We conclude that there exists $C>0$ independent of $f$ for which
\begin{equation}\label{2.10}
    \| \mathcal{G}_{\mathcal{H},\B}^\ell(f)\|_{L^p(\R,\gamma(H,\B))}
        \leq C \|f\|_{L^p(\R,\B)}.
\end{equation}

Our next objective is to establish that
$$\|f\|_{L^p(\mathbb{R}^n,\mathbb{B})}
    \leq C\|\mathcal{G}_{\mathcal{H},\B}^\ell(f)\|_{L ^p(\mathbb{R}^n,\gamma (H,\mathbb{B}))},$$
where $C>0$ does not depend on $f$.

In order to show this, we prove the following polarization formula.
\begin{Prop}\label{polarization}
    Let $\B$ be a UMD Banach space, $1<q<\infty$ and $k\in \mathbb{N}$.
    For every $f\in L^q(\mathbb{R}^n) \otimes \mathbb{B}$ and $g \in  L^{q'}(\mathbb{R}^n)\otimes \mathbb{B}^*$, we have that
    \begin{equation}\label{2.11}
        \int_{\mathbb{R}^n}\int_0^\infty \left\langle t^k\partial_t^k W_t^\mathcal{H}(g)(x),t^k\partial_t^k W_t^\mathcal{H}(f)(x)\right\rangle_{\mathbb{B}^*,\mathbb{B}}\frac{dt dx}{t}
            =\frac{\Gamma (2k)}{2^{2k}}\int_{\mathbb{R}^n}\langle g(x),f(x)\rangle_{\mathbb{B}^*,\mathbb{B}}dx.
    \end{equation}
\end{Prop}

\begin{proof}
    This property can be proved by using standard spectral arguments. Indeed, if $f,g\in \mbox{span}\{{\mathfrak h}_m\}_{m\in\mathbb{N}^n}$, we have that
    \begin{align*}
        \int_{\mathbb{R}^n}\int_0^\infty t^k\partial_t^k W_t^\mathcal{H}(g)(x)t^k\partial_t^k W_t^\mathcal{H}(f)(x)\frac{dt}{t}dx
            = \frac{\Gamma (2k)}{2^{2k}}\int_{\mathbb{R}^n}g(x)f(x)dx.
    \end{align*}
    Since $\mbox{span}\{{\mathfrak h}_m\}_{m\in \mathbb{N}^n}$ is dense in $L^q(\mathbb{R}^n)$, by taking into account that, for every $1<r<\infty$,
    $$\|g^k(\{W_t^\mathcal{H}\}_{t>0})(f)\|_{L^r(\mathbb{R}^n)}
        = \|\mathcal{G}_{\mathcal{H},\mathbb{C}}^k(f)\|_{L^r(\mathbb{R}^n,H)}
        \leq C\|f\|_{L^r(\mathbb{R}^n)},\quad f\in L^r(\mathbb{R}^n),$$
    we conclude that
    \begin{equation}\label{2.12}
        \int_{\mathbb{R}^n}\int_0^\infty t^k\partial_t^k W_t^\mathcal{H}(f)(x)t^k\partial_t^k W_t^\mathcal{H}(g)(x)\frac{dt}{t}
            =\frac{\Gamma (2k)}{2^{2k}}\int_{\mathbb{R}^n}f(x)g(x)dx,
    \end{equation}
    for every $f\in L^q(\mathbb{R}^n)$ and $g\in L^{q'}(\mathbb{R}^n)$.

    From \eqref{2.12} we can immediately deduce that \eqref{2.11} holds for every $L^p(\mathbb{R}^n)\otimes \mathbb{B}$ and
    $g\in L^{p'}(\mathbb{R}^n)\otimes \mathbb{B}^*$.
\end{proof}

Assume now that $F \in L^p(\mathbb{R}^n)\otimes \mathbb{B}$. According to \cite[Lemma 2.3]{GLY} we have that
$$\|F\|_{L^p(\mathbb{R}^n,\mathbb{B})}
    =\sup_{\substack{ g\in L^{p'}(\mathbb{R}^n)\otimes \mathbb{B}^* \\ \|g\|_{L^{p'}(\mathbb{R}^n,\mathbb{B}^*)}\leq 1 }}
            \left|\int_{\mathbb{R}^n} \langle g(x),F(x)\rangle _{\mathbb{B}^*,\mathbb{B}}dx\right|.$$
Then, since $\B^*$ is also a UMD Banach space, by using Proposition~\ref{polarization}, \cite[Proposition 2.2]{HW} and \eqref{2.10} we obtain
\begin{align*}
    \|F\|_{L^p(\mathbb{R}^n,\mathbb{B})}
        & = \frac{2^{2\ell}}{\Gamma (2 \ell)}\sup_{\substack{ g\in L^{p'}(\mathbb{R}^n)\otimes \mathbb{B}^* \\ \|g\|_{L^{p'}(\mathbb{R}^n,\mathbb{B}^*)}\leq 1 }}
                \left|\int_0^\infty \int_{\mathbb{R}^n}\langle \mathcal{G}_{\mathcal{H},\B^*}^\ell(g)(x,t),\mathcal{G}_{\mathcal{H},\B}^\ell(F)(x,t)\rangle _{\mathbb{B}^*,\mathbb{B}}dx\frac{dt}{t}\right|\\
        & \leq C \sup_{\substack{ g\in L^{p'}(\mathbb{R}^n)\otimes \mathbb{B}^* \\ \|g\|_{L^{p'}(\mathbb{R}^n,\mathbb{B}^*)}\leq 1 }}
                \int_{\mathbb{R}^n}\left\| \mathcal{G}_{\mathcal{H},\B^*}^\ell(g)(x,\cdot) \right\|_{\gamma (H,\mathbb{B}^*)}\left\| \mathcal{G}_{\mathcal{H},\B}^\ell(F)(x,\cdot) \right\|_{\gamma (H,\mathbb{B})}dx\\
        & \leq C \sup_{\substack{ g\in L^{p'}(\mathbb{R}^n)\otimes \mathbb{B}^* \\ \|g\|_{L^{p'}(\mathbb{R}^n,\mathbb{B}^*)}\leq 1 }}
                \left\| \mathcal{G}_{\mathcal{H},\B^*}^\ell(g) \right\|_{L^{p'}(\mathbb{R}^n,\gamma (H,\mathbb{B}^*))}\left\| \mathcal{G}_{\mathcal{H},\B}^\ell(F) \right\|_{L^p(\mathbb{R}^n, \gamma (H,\mathbb{B}))}\\
        & \leq C \left\|\mathcal{G}_{\mathcal{H},\B}^\ell(F) \right\|_{L^p(\mathbb{R}^n, \gamma (H,\mathbb{B}))}.
\end{align*}

By taking into account that $L^p(\mathbb{R}^n)\otimes \mathbb{B}$ is dense in $L^p(\mathbb{R}^n,\mathbb{B})$ and \eqref{2.10} we conclude that
$$\|f\|_{L^p(\mathbb{R}^n,\mathbb{B})}
    \leq C \left\|\mathcal{G}_{\mathcal{H},\B}^\ell(f)\right\|_{L^p(\mathbb{R}^n, \gamma (H,\mathbb{B}))},$$
where $C>0$ does not depend on $f$.

\subsection{} \fbox{$(b) \Rightarrow (a)$} Let $\gamma \in \mathbb{R}\setminus \{0\} $. We define the imaginary power $\mathcal{H}^{i\gamma}$
of $\mathcal{H}$ on $L^2(\R)$ as follows
$$\mathcal{H}^{i\gamma}f
    = \sum_{k \in \mathbb{N}^n} (2|k|+n)^{i\gamma} c_k(f){\mathfrak h}_k, \quad f \in L^2(\R).$$
Plancherel theorem implies that $\mathcal{H}^{i\gamma}$ is bounded from $L^2(\R)$ into itself. Moreover, $\mathcal{H}^{i\gamma}$ is an spectral multiplier
of Laplace transform type (\cite[p. 121]{Ste1}) associated with the Hermite operator and $\mathcal{H}^{i\gamma}$ can be
extended from $L^2(\R) \cap L^p(\R)$ to $L^p(\R)$ as a bounded operator from $L^p(\R)$ into itself, for every $1<p<\infty$
(\cite[Theorem 1.1]{BCCR}, \cite[Theorem 3]{BCFR}) .
Let $\B$ be a Banach space. If $1<p<\infty$ we can define in a natural way $\mathcal{H}^{i\gamma}$ on $L^p(\R)\otimes \B$ as a linear operator from
$L^p(\R)\otimes \B$ into itself.
In \cite[Theorem 1.2]{BCCR} (see also \cite[Theorem 3]{BCFR}) it was established that $\B$ is UMD if and only if
$\mathcal{H}^{i\gamma}$, $\gamma \in \mathbb{R}\setminus \{0\} $, can be extended from $L^p(\R)\otimes \B$ to $L^p(\R,\B)$ as a bounded operator
from $L^p(\R,\B)$ into itself for some (equivalently, for every) $1<p<\infty$.

Suppose now that $(b)$ holds. In order to see that $\mathbb{B}$ is UMD we prove the following
vector valued version of an inequality in \cite[p. 63]{Ste1}.

\begin{Prop}\label{G1G2}
    Let $\B$ be a Banach space and $\gamma \in \mathbb{R}\setminus \{0\}$.
    There exists $C>0$ such that, for every $f\in \mbox{span}\{{\mathfrak h}_k\}_{k\in \mathbb{N}^n}\otimes \B$,
    $$\|\mathcal{G}_{\mathcal{H},\B}^1(\mathcal{H}^{i\gamma}(f))(x,\cdot)\|_{\gamma (H,\mathbb{B})}
        \leq C\|\mathcal{G}_{\mathcal{H},\B}^2(f)(x,\cdot)\|_{\gamma (H,\mathbb{B})},\quad x\in \mathbb{R}^n. $$
\end{Prop}

\begin{proof}
    Let $f\in \mbox{span}\{{\mathfrak h}_k\}_{k\in \mathbb{N}^n}\otimes \B$. Then, $f=\sum_{k\in I}b_k{\mathfrak h}_k$, where $I$ is a finite subset  of
    $\mathbb{N}^n$ and $b_k\in \mathbb{B}$, $k\in I$.
    We introduce the operator $U\in L(\B)$ defined by $U(b)=-b$, $b\in \mathbb{B}$, and the operator $T_\gamma$ on $H$, given by
    $$T_\gamma (h)(t)
        =\frac{1}{t}\int_0^t\phi _\gamma (t-s)h(s)ds, \quad h\in H \mbox{ and }t>0,$$
    where $\phi _\gamma (u)=u^{-i\gamma}/\Gamma (1-i\gamma)$, $u>0$. The operator $T_\gamma \in L(H)$ and
    $\|T\|_{L(H)}\leq 1/\Gamma(1-i\gamma)$. Indeed, by using H\"older's inequality and Fubini's theorem we get
    \begin{align*}
        \|T_\gamma (h)\|_H
            &\leq \|\phi _\gamma \|_{L^\infty (0,\infty )}\left\{\int_0^\infty \left(\frac{1}{t}\int_0^t|h(s)|ds\right)^2\frac{dt}{t}\right\}^{1/2}\\
            &\leq \frac{1}{\Gamma(1-i\gamma)}\left\{\int_0^\infty |h(s)|^2\int_s^\infty\frac{dt}{t^2}ds\right\}^{1/2}
            = \frac{1}{\Gamma(1-i\gamma)} \|h\|_H,\quad h\in H.
    \end{align*}

    Let $x\in \mathbb{R}^n$. By considering $\mathcal{G}_{\mathcal{H},\B}^1(\mathcal{H}^{i\gamma}f)(x,\cdot)$ and $\mathcal{G}_{\mathcal{H},\B}^2(f)(x,\cdot)$
    as elements of $\gamma (H,\mathbb{B})$ we have that
    \begin{equation}\label{2.13}
        \mathcal{G}_{\mathcal{H},\B}^1(\mathcal{H}^{i\gamma}f)(x,\cdot)(h)
            =U\mathcal{G}_{\mathcal{H},\B}^2(f)(x,\cdot)T_\gamma (h),\quad h\in H.
    \end{equation}
    In fact, for every $h\in H$ and $S \in \B^*$, by using well-known properties of Laplace transform, we can write
    \begin{align*}
        \langle S, U \mathcal{G}_{\mathcal{H},\B}^2(f)(x,\cdot)T_\gamma (h) \rangle
            & = - \langle S, \mathcal{G}_{\mathcal{H},\B}^2(f)(x,\cdot)T_\gamma (h) \rangle \\
            & = - \int_0^\infty \langle S, \sum_{k \in I} b_k t^2 (2|k|+n)^2 e^{-t(2|k|+n)}{\mathfrak h}_k(x)  \rangle T_\gamma (h)(t)\frac{dt}{t}\\
            & = - \sum_{k \in I} \langle S,  b_k \rangle (2|k|+n)^2 {\mathfrak h}_k(x) \int_0^\infty  t  e^{-t(2|k|+n)} T_\gamma (h)(t)dt\\
            & = \left\langle S, - \sum_{k \in I}   b_k  (2|k|+n)^{i\gamma+1}{\mathfrak h}_k(x)  \int_0^\infty  e^{-t(2|k|+n)}h(t)dt \right\rangle.
    \end{align*}
    Hence,
    \begin{align*}
        U \mathcal{G}_{\mathcal{H},\B}^2(f)(x,\cdot)T_\gamma(h)
            &=-\sum_{k\in I}b_k (2|k|+n)^{i\gamma+1}{\mathfrak h}_k(x) \int_0^\infty e^{-t(2|k|+n)}h(t)dt, \quad h \in H.
    \end{align*}
    In a similar way we can see that
    \begin{align*}
        \mathcal{G}_{\mathcal{H},\B}^1(\mathcal{H}^{i\gamma}f)(x,\cdot)(h)
            &=-\sum_{k\in I}b_k (2|k|+n)^{i\gamma+1}{\mathfrak h}_k(x) \int_0^\infty e^{-t(2|k|+n)}h(t)dt, \quad h \in H.
    \end{align*}
    Thus \eqref{2.13} is established.

    By taking into account the ideal property for the $\gamma$-radonifying operators (\cite[Theorem 6.2]{Nee}) we conclude that
    $$\|\mathcal{G}_{\mathcal{H},\B}^1(\mathcal{H}^{i\gamma}(f))(x,\cdot)\|_{\gamma (H,\mathbb{B})}
        \leq \frac{1}{\Gamma(1-i\gamma)}\|\mathcal{G}_{\mathcal{H},\B}^2(f)(x,\cdot)\|_{\gamma (H,\mathbb{B})}.$$
\end{proof}

Let $\gamma \in \mathbb{R} \setminus \{0\}$ and $1<p<\infty$.
From $(b)$ and Proposition~\ref{G1G2} it follows, for every $f\in \mbox{span}\{{\mathfrak h}_k\}_{k\in \mathbb{N}^n}\otimes \B$,
$$\|\mathcal{H}^{i\gamma}(f)\|_{L^p(\mathbb{R}^n,\mathbb{B})}
    \leq C\|\mathcal{G}_{\mathcal{H},\B}^1(\mathcal{H}^{i\gamma}(f))\|_{L^p(\mathbb{R}^n,\gamma (H,\mathbb{B}))}
    \leq C \|\mathcal{G}_{\mathcal{H},\B}^2(f)\|_{L^p(\mathbb{R}^n,\gamma (H,\mathbb{B}))}
    \leq C\|f\|_{L^p(\mathbb{R}^n,\mathbb{B})}.$$
Since $\mbox{span}\{{\mathfrak h}_k\}_{k\in \mathbb{N}^n}\otimes \B$ is a dense subspace in  $L^p(\mathbb{R}^n,\mathbb{B})$,
$\mathcal{H}^{i\gamma}$ can be extended to $L^p(\mathbb{R}^n,\mathbb{B})$ as a bounded operator from $L^p(\mathbb{R}^n,\mathbb{B})$ into itself.
From \cite[Theorem 1.2]{BCCR} (\cite[Theorem 3]{BCFR}) we deduce that $\B$ is UMD.

\section{Proof of Theorem~\ref{mean} for the Schr\"odinger operator}\label{sec:Schorindger}

In this section we prove $(a) \Leftrightarrow (c)$ in Theorem~\ref{mean}. We assume that $n\geq 3$ and that the potential function $V$ satisfies the reverse
H\"older's inequality \eqref{RH} where $s>n/2$. In the proof of $(c)\Rightarrow (a)$ we will use
\cite[Theorem 3]{BCFR} where UMD Banach spaces are characterized
by the $L^p$-boundedness properties of the imaginary power $\mathcal{L}^{i\gamma}$, $\gamma\in \mathbb{R}\setminus \{0\}$, of the Schr\"odinger operator
$\mathcal{L}$.

\subsection{} \fbox{$(a) \Rightarrow (c)$} In order to show this result we can proceed as in the proof of $(a)\Rightarrow (b)$ by using in each moment
the suitable property for the heat kernel $W_t^\mathcal{L}(x,y)$, $x,y\in \mathbb{R}^n$ and $t>0$, of the Sch\"odinger semigroup.

As it is showed in the papers of Dziuba\'nski and Zienkiewicz (\cite{DZ1}, \cite{DZ2} and \cite{DZ3}),
Dziuba\'nski, Garrig\'os, Mart\'{\i}nez, Torrea  and Zienkiewicz (\cite{DGMTZ}) and Shen (\cite{Sh}), the function $\rho$ defined by
$$\rho (x)
    =\sup \left\{r>0:\frac{1}{r^{n-2}}\int_{B(x,r)}V(y)dy\leq 1\right\},\quad x\in \mathbb{R}^n,$$
plays an important role in the develop of the harmonic analysis in the Schr\"odinger setting. In the special case of the Hermite operator, we can see that
$$\rho (x)
    \sim \left\{\begin{array}{ll}
                \displaystyle \frac{1}{2},&|x|\leq 1,\\
                \displaystyle \frac{1}{1+|x|},&|x|\geq 1  .
             \end{array}
\right.$$
This function $\rho$ is usually called ``critical radius" of $x$, and we use it to split the operators in the local and global parts (see \eqref{7.1}).
The main properties of the function $\rho$ can be encountered in \cite[Lemma 1.4]{Sh}. We must apply repeatedly  that, for every $M>0$, $\rho (x)\sim \rho (y)$
provided that $x,y\in \mathbb{R}^n$ and $|x-y|\leq M\rho (x)$, where the equivalence constants depend only on $M$.
Also, according to \cite[Proposition 5]{DGMTZ},
we can find a sequence $\{x_k\}_{k\in \mathbb{N}}\subset \mathbb{R}^n$ such that:

$(i)$ $\overset{}{\underset{k\in \mathbb{N}}{\bigcup}} B(x_k,\rho (x_k))=\mathbb{R}^n$;

$(ii)$ for every $M>0$ there exists $m\in \mathbb{N}$ such that, for every $k\in \mathbb{N}$,
$$ \mbox{card }\{j\in\mathbb{N}: B(x_j,M\rho  (x_j))\cap B(x_k,M\rho (x_k))\not=\emptyset\}\leq m.$$
To complete the proof we need to use the following properties of $W_t^\mathcal{L}(x,y)$, $x,y\in \mathbb{R}^n$ and $t>0$.
All of them can be found, for instance in \cite[Section 2]{DGMTZ} and \cite[Section 2]{DZ2}.

\begin{Lem}\label{Lem Sch}
    Assume that $V\in RH_s$, where $s>n/2$. Then,

    $(i)$ For every $k ,N\in \mathbb {N}$, there exist $C,c>0$ for which
    $$\left|t^k \partial_t^k W_t^\mathcal{L}(x,y)\right|
        \leq C\frac{e^{-c|x-y|^2/t}}{t^{n/2}}\left(1+\frac{\sqrt{t}}{\rho (x)}+\frac{\sqrt{t}}{\rho (y)}\right)^{-N},\quad x,y\in \mathbb{R}^n\mbox{ and }t>0. $$

    $(ii)$ There exists a nonnegative function $w\in S(\mathbb{R}^n)$, the Schwartz functions space, and $\delta >0$, such that
    $$|G_{\sqrt{t}}(x-y)-W_t^\mathcal{L}(x,y)|
        \leq C\left(\frac{\sqrt{t}}{\rho (x)}\right)^\delta w_{\sqrt{t}}(x-y),\quad 0<t\leq \rho (x)^2, \ x,y\in \mathbb{R}^n.$$

    $(iii)$ If $w\in S(\mathbb{R}^n)$ there exist $\delta ,\beta >0$ such that
    $$\int_{\mathbb{R}^n}G_{\sqrt{t}}(x-y)V(y)dy
        \leq C\left\{\begin{array}{ll}
                        \displaystyle \frac{1}{t}\left(\frac{\sqrt{t}}{\rho (x)}\right)^\delta ,&0<t\leq \rho (x)^2,\\
                        &\\
                        \displaystyle \left(\frac{\sqrt{t}}{\rho (x)}\right)^{\beta +2-n},&t>\rho (x)^2.
                        \end{array}
    \right.$$
\end{Lem}
The polarization equality (see \eqref{2.11}) can be shown in the Schr\"odinger setting by using spectral arguments.

\subsection{} \fbox{$(c)\Rightarrow (a)$} Assume that $(c)$ holds for a certain $1<p<\infty$.

We denote by $E_\mathcal{L}(d\lambda )$ the spectral measure associated to the Schr\"odinger operator $\mathcal{L}$ . Then, we have that
$$W_t^\mathcal{L}(f)
    =\int_{[0,\infty )}e^{-\lambda t}E_\mathcal{L}(d\lambda)f,\quad f\in L^2(\mathbb{R}^n).$$
We can also write
$$W_t^\mathcal{L}(f)(x)
    =\int_{\mathbb{R}^n}W_t^\mathcal{L}(x,y)f(y)dy,\quad f\in L^2(\mathbb{R}^n)\mbox{ and }x\in \mathbb{R}^n.$$
Let $f,g\in L^2(\mathbb{R}^n)$. Then,
\begin{align*}
    \langle \partial_t W_t^\mathcal{L}(f)(x),g(x)\rangle
        &=\int_{\mathbb{R}^n}\int_{\mathbb{R}^n}\partial_t W_t^\mathcal{L}(x,y)f(y)dyg(x)dx
        =\partial_t \langle W_t^\mathcal{L}f(x),g(x)\rangle,\;\;\; x\in \mathbb{R}^n \mbox{ and }t>0.
\end{align*}
Note that by using Lemma~\ref{Lem Sch}, $(i)$, we can justified the derivation under the integral sign.
Indeed, Lemma~\ref{Lem Sch}, $(i)$, implies that
\begin{align*}
    & \int_{\mathbb{R}^n}\int_{\mathbb{R}^n}\left|\partial_t W_t^\mathcal{L}(x,y)\right\| |f(y)| |g(x)|dydx
        \leq C\int_{\mathbb{R}^n}\int_{\mathbb{R}^n}\frac{e^{-c|x-y|^2/t}}{t^{n/2 +1}}|f(y)\|g(x)|dydx\\
    & \qquad \qquad \leq \frac{C}{t}\int_{\mathbb{R}^n}\sup_{s>0} \left( \int_{\mathbb{R}^n}\frac{e^{-c|x-y|^2/s}}{s^{n/2}}|f(y)|dy \right) |g(x)|dx\\
    & \qquad \qquad \leq \frac{C}{t} \|g\|_{L^2(\mathbb{R}^n)} \left\{\int_{\mathbb{R}^n}\left(\sup_{s>0}\int_{\mathbb{R}^n}\frac{e^{-c|x-y|^2/s}}{s^{n/2}}|f(y)|dy\right)^2dx\right\}^{1/2}
    \leq \frac{C}{t},\quad t>0,
\end{align*}
because the maximal operator $W_*$ defined by
$$W_*(F)
    =\sup_{s>0}|W_s(F)|,\quad F\in L^2(\mathbb{R}^n),$$
is bounded from $L^2(\mathbb{R}^n)$ into itself.

On the other hand, by defining
$$D_tW_t^\mathcal{L}(f)
    =\lim_{h\rightarrow 0}\frac{W_{t+h}^\mathcal{L}(f)-W_t^\mathcal{L}(f)}{h},\quad \mbox{ on }L^2(\mathbb{R}^n),$$
we have that
$$D_tW_t^\mathcal{L}(f)
    =-\int_{[0,\infty )}\lambda e^{-\lambda t}E_\mathcal{L}(d\lambda )f,\quad t>0.$$
Hence, we conclude that, for every $t>0$,
$$D_tW_t^\mathcal{L}(f)(x)
    =\int_{\mathbb{R}^n}\partial_t W_t^\mathcal{L}(x,y)f(y)dy,\quad \mbox{ a.e. }x\in \mathbb{R}^n.$$
Then, for every $f\in L^2(\mathbb{R}^n)\otimes \mathbb{B}$, $\ell =1,2$, and $t>0$,
$$\mathcal{G}_{\mathcal{L},\mathbb{B}}^\ell (f)(\cdot ,t)
    =\int_{[0,\infty )}(-\lambda t)^\ell e^{-\lambda t}E_\mathcal{L}(d\lambda )f,$$
where the right hand side has the obvious meaning.

Let $\gamma \in \mathbb{R}\setminus\{0\}$. The imaginary power $\mathcal{L}^{i\gamma }$ of the operator $\mathcal{L}$ is defined by
$$\mathcal{L}^{i\gamma }(f)
    =\int_{[0,\infty )}\lambda ^{i\gamma }E_\mathcal{L}(d\lambda )f,\quad f\in L^2(\mathbb{R}^n),$$
and we extend $\mathcal{L}^{i\gamma }$ to $L^2(\mathbb{R}^n)\otimes \mathbb{B}$ in the natural way.

It is clear that, for every $t>0$,
$$\mathcal{G}_{\mathcal{L},\mathbb{B}}^1 (\mathcal{L}^{i\gamma }f)(\cdot ,t)
    =-\int_{[0,\infty )}t\lambda^{1+i\gamma }e^{-\lambda t}E_\mathcal{L}(d\lambda )f, \quad f\in L^2(\mathbb{R}^n)\otimes \mathbb{B}.$$

In the following we establish the analogous property of Proposition~\ref{G1G2} but in the Schr\"odinger setting.

\begin{Prop}\label{G1G2 Schr}
    Let $\B$ be a Banach space and $\gamma \in \mathbb{R}\setminus \{0\}$.
    There exists $C>0$ such that, for every $f\in S(\R)\otimes \B$,
    $$\|\mathcal{G}_{\mathcal{L},\B}^1(\mathcal{L}^{i\gamma}(f))(x,\cdot)\|_{\gamma (H,\mathbb{B})}
        \leq C\|\mathcal{G}_{\mathcal{L},\B}^2(f)(x,\cdot)\|_{\gamma (H,\mathbb{B})},\quad \text{a.e. }x\in \mathbb{R}^n. $$
\end{Prop}

\begin{proof}
    Let $h\in L^2((0,\infty ),dt/t)$ such that $\mbox{supp}\;h\subset (a,b)$, $0<a<b<\infty $, and let $f,g\in L^2(\mathbb{R}^n)$.
    According to Lemma~\ref{Lem Sch}, $(i)$, we get as above
    \begin{align*}
        & \int_0^\infty \int_{\mathbb{R}^n}|\mathcal{G}_{\mathcal{L},\mathbb{B}}^1 (\mathcal{L}^{i\gamma }f)(x,t)g(x)|dx|h(t)|\frac{dt}{t}\\
        & \qquad \qquad \leq C\int_0^\infty \int_{\mathbb{R}^n}\sup _{s>0} \left(\int_{\mathbb{R}^n}\frac{e^{-c|x-y|^2/t}}{s^{n/2}}|\mathcal{L}^{i\gamma }f(y)|dy \right)|g(x)|dx|h(t)|\frac{dt}{t}\\
        & \qquad \qquad \leq C\|g\|_{L^2(\mathbb{R}^n)}\|W_*(\mathcal{L}^{i\gamma }f)\|_{L^2(\mathbb{R}^n)}\int_a^b|h(t)|\frac{dt}{t}<\infty .
    \end{align*}

    We can write
    \begin{align*}
        & \int_{\mathbb{R}^n}\int_0^\infty \mathcal{G}_{\mathcal{L},\mathbb{B}}^1 (\mathcal{L}^{i\gamma }f)(x,t)h(t)\frac{dt}{t}g(x)dx\\
        &\qquad \qquad = \int_0^\infty h(t)\int_{\mathbb{R}^n}\mathcal{G}_{\mathcal{L},\mathbb{B}}^1 (\mathcal{L}^{i\gamma }f)(x,t)g(x) dx\frac{dt}{t}\\
        &\qquad \qquad =-\int_0^\infty h(t)\int_{\mathbb{R}^n}\left(\int_{[0,\infty )}t\lambda ^{1+i\gamma }e^{-\lambda t}E_\mathcal{L}(d\lambda )f\right)(x)g(x)dx\frac{dt}{t}\\
        &\qquad \qquad =-\int_0^\infty \int_{[0,\infty )}\lambda ^{1+i\gamma }e^{-\lambda t}d\mu _{f,g}(\lambda )h(t)dt,
    \end{align*}
    where $\mu _{f,g}$ represents the complex measure defined by
    $$\mu _{f,g}(A)
        =\langle E_\mathcal{L}(A)f,g\rangle,$$
    for every  Borel set $A$ in $[0,\infty )$. If $|\mu _{f,g}|$ denotes the total variation measure of $\mu _{f,g}$, then
    $$\int_0^\infty \int_{[0,\infty )}|\lambda ^{1+i\gamma }|e^{-\lambda t}d|\mu _{f,g}|(\lambda )|h(t)|dt
        \leq C|\mu _{f,g}|([0,\infty ))\int_0^\infty |h(t)|\frac{dt}{t}<\infty .$$
    Hence, we have that
    \begin{align*}
        & \int_{\mathbb{R}^n}\int_0^\infty \mathcal{G}_{\mathcal{L},\mathbb{B}}^1 (\mathcal{L}^{i\gamma }f)(x,t)h(t)\frac{dt}{t}g(x)dx\\
        &\qquad \qquad =-\int_{[0,\infty )}\lambda ^{1+i\gamma }\int_0^\infty e^{-\lambda t}h(t)dtd\mu _{f,g}(\lambda )\\
        &\qquad \qquad =-\int_{[0,\infty )}\lambda ^2\int_0^\infty t^2e^{-\lambda t}T_\gamma (h)(t)\frac{dt}{t}d\mu _{f,g}(\lambda ),
    \end{align*}
    where
    $$T_\gamma (h)(t)=\frac{1}{t}\int_0^t\phi _\gamma (t-s)h(s)ds, \quad t\in (0,\infty ),$$
    and $\phi _\gamma (u)=u^{-i\gamma}/\Gamma (1-i\gamma)$, $u\in (0,\infty )$. Since
    $$\int_{[0,\infty )}\int_0^\infty(\lambda t)^2e^{-\lambda t}|T_\gamma (h)(t)|\frac{dt}{t}d|\mu _{f,g}|(\lambda )
        <\infty,$$
    we can write
    \begin{align*}
        \int_{\mathbb{R}^n}\int_0^\infty \mathcal{G}_{\mathcal{L},\mathbb{B}}^1 (\mathcal{L}^{i\gamma }f)(x,t)h(t)\frac{dt}{t}g(x)dx
            &=-\int_0^\infty T_\gamma (h)(t)\int_{[0,\infty )}(\lambda t)^2e^{-\lambda t}d\mu _{f,g}(\lambda )\frac{dt}{t}\\
        & =-\int_0^\infty T_\gamma (h)(t)\int_{\mathbb{R}^n}g(x)\mathcal{G}_{\mathcal{L},\mathbb{C}}^2(f)(x,t)dx\frac{dt}{t}\\
        &  =-\int_{\mathbb{R}^n}g(x)\int_0^\infty T_\gamma (h)(t)\mathcal{G}_{\mathcal{L},\mathbb{C}}^2(f)(x,t)\frac{dt}{t}dx.
    \end{align*}
    The last interchange is justified because
    $$\int_0^\infty \int_{\mathbb{R}^n}|g(x)\|\mathcal{G}_{\mathcal{L},\mathbb{C}}^2(f)(x,t)|dxT_\gamma (h)(t)\frac{dt}{t}
        \leq C\|h\|_H\int_a^\infty \int_{\mathbb{R}^n}|g(x)|W_*(|f|)(x)dx\frac{dt}{t^2}<\infty.$$
    We have taken into account that, since ${\rm supp}\;h\subset (a,b)$, it follows that $T_\gamma (h)(t)=0$, when $t\in (0,a)$.

    We conclude that
    $$\int_0^\infty \mathcal{G}_{\mathcal{L},\mathbb{C}}^1 (\mathcal{L}^{i\gamma }f)(x,t)h(t)\frac{dt}{t}
        =-\int_0^\infty T_\gamma (h)(t)\mathcal{G}_{\mathcal{L},\mathbb{C}}^2(f)(x,t)\frac{dt}{t},\quad \mbox{a.e. }x\in \mathbb{R}^n.$$
    It is well-known that the space $C_c(0,\infty )$ of continuous functions with compact support is dense in $H$.
    Moreover, since $H$ is separable, there exists a numerable set $\mathcal{A}\subset C_c(0,\infty )$ that is dense in $H$.

    We define ${\bf N}\subset \mathbb{R}^n$ consisting on those $x\in\mathbb{R}^n$ for which
    $$\int_0^\infty \mathcal{G}_{\mathcal{L},\mathbb{C}}^1 (\mathcal{L}^{i\gamma }f)(x,t)h(t)\frac{dt}{t}
        =-\int_0^\infty T_\gamma (h)(t)\mathcal{G}_{\mathcal{L},\mathbb{C}}^2(f)(x,t)\frac{dt}{t},\quad h\in \mathcal{A}.$$
    We have that $|\mathbb{R}^n\setminus {\bf N}|=0$. then, for every $h\in H$,
    $$\int_0^\infty \mathcal{G}_{\mathcal{L},\mathbb{C}}^1 (\mathcal{L}^{i\gamma }f)(x,t)h(t)\frac{dt}{t}
        =-\int_0^\infty T_\gamma (h)(t)\mathcal{G}_{\mathcal{L},\mathbb{C}}^2(f)(x,t)\frac{dt}{t},\quad x\in {\bf N} .$$
    Hence, if $f\in L^2(\mathbb{R}^n)\otimes \mathbb{B}$, there exists $\Omega \subset \mathbb{R}^n$ such that $|\mathbb{R}^n\setminus {\bf N}|=0$ and
    $$\int_0^\infty \mathcal{G}_{\mathcal{L},\mathbb{B}}^1 (\mathcal{L}^{i\gamma }f)(x,t)h(t)\frac{dt}{t}
        =-\int_0^\infty T_\gamma (h)(t)\mathcal{G}_{\mathcal{L},\mathbb{B}}^2(f)(x,t)\frac{dt}{t},\quad x\in \Omega ,$$
    for every $h\in H$. By defining $Ub=-b$, $b\in \mathbb{B}$, we have that, as elements of $\gamma (H,\mathbb{B})$,
    $$\mathcal{G}_{\mathcal{L},\mathbb{B}}^1(\mathcal{L}^{i\gamma }f)(x,\cdot)
        =U\mathcal{G}_{\mathcal{L},\mathbb{B}}^2(f)(x,\cdot)T_\gamma ,\quad \mbox{a.e. }x\in \mathbb{R}^n,$$
    for every $f\in S(\mathbb{R}^n)\otimes \mathbb{B}$.

    By taking into account the ideal property of $\gamma(H,\mathbb{B})$ (\cite[Theorem 6.2]{Nee}), and that the operators $U$ and $T_\gamma$
    are bounded in $\B$ and $H$, respectively,  we conclude the proof of this proposition.
\end{proof}

Finally, from the equivalences in $(c)$ and Proposition~\ref{G1G2 Schr}, we have that, for every $f\in S(\mathbb{R}^n)\otimes \mathbb{B}$,
\begin{align*}
    \|\mathcal{L}^{i\gamma }\|_{L^p(\mathbb{R}^n,\mathbb{B})}
        &\leq C\|\mathcal{G}_{\mathcal{L},\mathbb{B}}^1(\mathcal{L}^{i\gamma }f)\|_{L^p(\mathbb{R}^n,\gamma (H,\mathbb{B}))}
        \leq C\|\mathcal{G}_{\mathcal{L},\mathbb{B}}^2(f)\|_{L^p(\mathbb{R}^n,\gamma (H,\mathbb{B}))}\leq C\|f\|_{L^p(\mathbb{R}^n,\mathbb{B})}.
\end{align*}

Since $S(\mathbb{R}^n)\otimes \mathbb{B}$ is dense in $L^p(\mathbb{R}^n,\mathbb{B})$, we have proved that the operator
$\mathcal{L}^{i\gamma }$ can be extended to $L^p(\mathbb{R}^n,\mathbb{B})$ as a bounded operator from $L^p(\mathbb{R}^n,\mathbb{B})$ into itself.
Then, according to \cite[Theorem 3]{BCFR}, $\mathbb{B}$ is UMD.

\section{Proof of Theorem~\ref{mean} for Laguerre operators}\label{sec:Laguerre}

In this section we prove the equivalence $(a)\Leftrightarrow (d)$ in Theorem~\ref{mean}.

Suppose that $\mathbb{B}$ is a UMD Banach space. Let $\ell =1,2$ and $1<p<\infty$.

We are going to see that
\begin{equation}\label{gLaguerre}
    \|\mathcal{G}_{\mathcal{L}_\alpha ,\mathbb{B}}^\ell (f)\|_{L^p((0,\infty ),\gamma (H,\mathbb{B}))}
        \leq C\|f\|_{L^p((0,\infty ),\mathbb{B})},\quad f\in L^p((0,\infty ),\mathbb{B}).
\end{equation}

In order to show \eqref{gLaguerre} we take advantage from $(a)\Rightarrow (b)$ after connecting
$\mathcal{G}_{\mathcal{L}_\alpha ,\mathbb{B}}^\ell $ and $\mathcal{G}_{\mathcal{H},\mathbb{B}}^\ell $ in a suitable way.

Note firstly that
\begin{equation}\label{20.1}
    W_t^{\mathcal{L}_\alpha}(x,y)
        =W_t^{\mathcal{H}/2}(x,y)g_\alpha \left(\frac{2xye^{-t}}{1-e^{-2t}}\right),\quad x,y,t\in (0,\infty ),
\end{equation}
where $W_t^{\mathcal{H}/2}(x,y)$ denotes the heat kernel associated with the operator $\mathcal{H}/2$ in dimension one,
that is, for each $x,y\in \mathbb{R}\mbox{ and }t>0$,
$$W_t^{\mathcal{H}/2}(x,y)
    =\left(\frac{e^{-t}}{\pi (1-e^{-2t})}\right)^{1/2}\exp\left[-\frac{1}{4}\left((x-y)^2\frac{1+e^{-t}}{1-e^{-t}}+(x+y)^2\frac{1-e^{-t}}{1+e^{-t}}\right)\right],$$
and $g_\alpha$ is defined by
$$g_\alpha (z)
    =\sqrt{2\pi z}e^{-z}I_\alpha (z), \quad z\in (0,\infty ).$$

To make the reading of the following lines easier, from now on we consider $\xi =\xi(x,y,t)=\frac{2xye^{-t}}{1-e^{-2t}}$,  $x,y,t\in (0,\infty )$.

We have, for every $x,y,t\in (0,\infty )$,
\begin{equation}\label{deriv1}
    \partial_t W_t^{\mathcal{L}_\alpha} (x,y)
        =\partial_t W_t^{\mathcal{H}/2}(x,y)g_\alpha (\xi)
            -W_t^{\mathcal{H}/2}(x,y)\Big(\frac{d}{dz}g_\alpha(z)\Big)_{\big|z=\xi }\frac{\xi (1+e^{-2t})}{1-e^{-2t}},
\end{equation}
and
\begin{align}\label{deriv2}
    \partial_t^2 W_t^{\mathcal{L}_\alpha }(x,y)
        =&\partial_t^2 W_t^{\mathcal{H}/2}(x,y)g_\alpha (\xi)-2\partial_t W_t^{\mathcal{H}/2}(x,y)\Big(\frac{d}{dz}g_\alpha(z)\Big)_{\big|z=\xi }\frac{\xi (1+e^{-2t})}{1-e^{-2t}}\nonumber\\
        &+W_t^{\mathcal{H}/2}(x,y)\left\{\Big(\frac{d^2}{dz^2}g_\alpha(z)\Big)_{\big|z=\xi }\frac{\xi^2(1+e^{-2t})^2}{(1-e^{-2t})^2}+\Big(\frac{d}{dz}g_\alpha(z)\Big)_{\big|z=\xi }\frac{\xi(1+6e^{-2t}+e^{-4t})}{(1-e^{-2t})^2}\right\}.
\end{align}

By taking into account that $\frac{d}{dz}(z^{-\alpha }I_\alpha (z))=z^{-\alpha }I_{\alpha +1}(z)$, $z\in (0,\infty )$ (\cite[p. 110]{Leb}), we get
\begin{equation}\label{dg}
    \frac{d}{dz}g_\alpha (z)
        =-g_\alpha (z)+\frac{2\alpha +1}{2z}g_\alpha (z)+g_{\alpha +1}(z),\quad z\in (0,\infty ),
\end{equation}
and
\begin{equation}\label{d2g}
    \frac{d^2}{dz^2}g_\alpha (z)
        =\left(1-\frac{2\alpha +1}{z}+\frac{4\alpha^2-1}{4z^2}\right)g_\alpha (z)+2\left(\frac{\alpha -1}{z}-1\right)g_{\alpha +1}(z)+g_{\alpha +2}(z),
        \ z\in (0,\infty ).
\end{equation}

Since  $I_\alpha (z)\sim z^\alpha $, as $z\rightarrow 0^+$ (\cite[p. 108]{Leb}), we deduce from \eqref{dg} and \eqref{d2g} that, for $k=0,1,2$,
\begin{equation}\label{g1}
    \Big|z^k\frac{d^k}{dz^ k}g_\alpha (z)\Big|
        \leq C, \quad z\in (0,1).
\end{equation}
On the other hand, according to \cite[p. 123]{Leb}, for every $m\in \mathbb{N}$,
\begin{equation}\label{asimptoticgalpha}
    g_\alpha (z)
        =\sum_{r=0}^m \frac{(-1)^r[\alpha ,r]}{(2z)^{r}}+\mathcal{O}\left(\frac{1}{z^{m+1}}\right), \quad z\in (0,\infty),
\end{equation}
where $[\alpha ,0]=1$ and
$$[\alpha ,r]
    =\frac{(4\alpha ^2-1)(4\alpha ^2-3^2)\cdots (4\alpha ^2-(2r-1)^2)}{2^{2r}\Gamma (r+1)},\quad r\in \mathbb{N}, \ r\geq 1 .$$
Then, for $k=1,2$,
\begin{equation}\label{g2}
    \Big|z^k\frac{d^k}{dz^ k}g_\alpha (z)\Big|
        \leq \frac{C}{z}, \quad z\in (0,\infty ).
\end{equation}
Indeed, by using property \eqref{asimptoticgalpha} in \eqref{dg} and \eqref{d2g} we get
$$\frac{d}{dz}g_\alpha (z)
    =\frac{[\alpha ,1]-[\alpha +1,1]+2\alpha +1}{2z}+\mathcal{O}\Big(\frac{1}{z^2}\Big)=\mathcal{O}\Big(\frac{1}{z^2}\Big),\quad z\in (0,\infty).$$
and
\begin{align*}
    \frac{d^2}{dz^2}g_\alpha (z)
        =&\left(1-\frac{[\alpha ,1]}{2}+[\alpha +1,1]-\frac{[\alpha +2,1]}{2}\right)\frac{1}{z}\\
        &+\left(\frac{[\alpha ,2]}{4}+(2\alpha +3)\frac{[\alpha ,1]}{2}-(\alpha+1)[\alpha +1,1]-\frac{[\alpha+1,2]}{2}+\frac{[\alpha +2,2]}{4}\right)\frac{1}{z^2}\\
        &+\mathcal{O}\Big(\frac{1}{z^3}\Big)=\mathcal{O}\Big(\frac{1}{z^3}\Big),\quad z\in (0,\infty).
\end{align*}
Then, \eqref{g2} holds for $k=1,2$.

From \eqref{deriv1} and \eqref{deriv2} by using \eqref{2.4} (note that estimate \eqref{2.4} also holds for $\mathcal{H}/2$ instead of $\mathcal{H}$) we obtain,
for every $x,y,t\in (0,\infty )$,
\begin{align}\label{principal}
    & \left|t^\ell\partial_t^\ell W_t^{\mathcal{L}_\alpha }(x,y)\right|
        \leq  C \Big\{ \left|t^\ell\partial_t^\ell W_t^{\mathcal{H}/2}(x,y)\right| \left| g_\alpha(\xi) \right|
                            +\left|t^\ell\partial_t^{\ell-1} W_t^{\mathcal{H}/2}(x,y)\right|\left|\Big(\frac{d}{dz}g_\alpha (z)\Big)_{\big|z=\xi}\right|\frac{\xi}{1-e^{-2t}}\nonumber\\
    & \qquad  \qquad +(\ell-1)t^2 W_t^{\mathcal{H}/2}(x,y)\left[\left|\Big(\frac{d^2}{dz^2}g_\alpha (z)\Big)_{\big|z=\xi}\right|\frac{\xi ^2}{(1-e^{-2t})^2}+\left|\Big(\frac{d}{dz}g_\alpha (z)\Big)_{\big|z=\xi}\right|\frac{\xi}{(1-e^{-2t})^2}\right] \Big\}\nonumber\\
    & \qquad \leq C\frac{e^{-t/3}e^{-c|x-y|^2/t}}{\sqrt{t}}
            \left\{ \left| g_\alpha(\xi) \right|
                        +\frac{t\xi}{1-e^{-2t}} \left|\Big(\frac{d}{dz}g_\alpha (z)\Big)_{\big|z=\xi}\right|
                        +(\ell -1)\frac{(t\xi )^2}{(1-e^{-2t})^2}\left|\Big(\frac{d^2}{dz^2}g_\alpha (z)\Big)_{\big|z=\xi}\right|\right\}.
\end{align}

Now,  \eqref{g1} implies that
\begin{equation}\label{B3menor}
    \left|t^\ell\partial_t^\ell W_t^{\mathcal{L}_\alpha }(x,y)\right|\leq C \frac{e^{-c|x-y|^2/t}}{\sqrt{t}},\quad x,y,t\in (0,\infty )\mbox{  and } \xi \leq 1.
\end{equation}
We also observe that
\begin{align*}
    & \exp\left[-c\left(\frac{1+e^{-t}}{1-e^{-t}}|x-y|^2+\frac{1-e^{-t}}{1+e^{-t}}|x+y|^2\right)\right]
        = \exp\left[-2c\left(\frac{1+e^{-2t}}{1-e^{-2t}}(x^2+y^2)-8\frac{e^{-t}}{1-e^{-2t}}xy\right)\right]\\
    & \qquad \qquad \leq \exp\left[-c\frac{1+e^{-2t}}{1-e^{-2t}}(x^2+y^2)\right],\quad x,y,t\in (0,\infty )\mbox{ and }\xi \leq 1.
\end{align*}
Then we get
\begin{equation}\label{B3bisH}
    \left|t^\ell \partial_t^\ell W_t^{\mathcal{H}/2}(x,y)\right|
        \leq C\frac{e^{-c(x^2+y^2)/t}}{\sqrt{t}},\quad x,y,t\in (0,\infty ), \ \xi \leq 1,
\end{equation}
and
\begin{equation}\label{B3bis}
    \left|t^\ell\partial_t^\ell W_t^{\mathcal{L}_\alpha }(x,y)\right|
        \leq C\frac{e^{-c(x^2+y^2)/t}}{\sqrt{t}},\quad x,y,t\in (0,\infty )\mbox{  and } \xi \leq 1.
\end{equation}

On the other hand, from \eqref{g2} and \eqref{principal} it follows that,
\begin {equation}\label{B3mayor}
    \left|t^\ell\partial_t^\ell W_t^{\mathcal{L}_\alpha }(x,y)\right|
        \leq C\frac{xye^{-c|x-y|^2/t}}{t^{3/2}},\quad x,y,t\in (0,\infty )\mbox{ and }\xi \geq 1.
\end{equation}
Moreover, by taking into account \eqref{2.4}, \eqref{20.1}, \eqref{asimptoticgalpha} and \eqref{g2}, we obtain that
\begin{align*}
    & \left|t^\ell\partial_t^\ell \left[W_t^{\mathcal{L}_\alpha }(x,y)-W_t^{\mathcal{H}/2}(x,y)\right]\right|\\
        & \qquad \leq \left|t^\ell\partial_t^\ell [W_t^{\mathcal{H}/2}(x,y)](g_\alpha (\xi )-1)\right|+\left|t^\ell\partial_t^{\ell-1} W_t^{\mathcal{H}/2}(x,y)\right|\left|\Big(\frac{d}{dz}g_\alpha (z)\Big)_{\big|z=\xi}\right|\frac{\xi}{1-e^{-2t}}\\
        & \qquad \qquad +(\ell-1)t^2 W_t^{\mathcal{H}/2}(x,y)\left\{\left|\Big(\frac{d^2}{dz^2}g_\alpha (z)\Big)_{\big|z=\xi}\right|\frac{\xi ^2}{(1-e^{-2t})^2}+\left|\Big(\frac{d}{dz}g_\alpha (z)\Big)_{\big|z=\xi}\right|\frac{\xi}{(1-e^{-2t})^2}\right\}\\
        & \qquad \leq C\frac{e^{-t/3}e^{-c|x-y|^2/t}}{\xi\sqrt{t}},\quad x,y,t\in (0,\infty ).
\end{align*}

Hence, we get
\begin{equation}\label{B4}
    \left|t^\ell\partial_t^\ell [W_t^{\mathcal{L}_\alpha }(x,y)-W_t^{\mathcal{H}/2}(x,y)]\right|
        \leq C\frac{e^{-t/3}e^{-c|x-y|^2/t}}{\xi^{1/4}\sqrt{t}}\leq C\frac{e^{-c|x-y|^2/t}}{(xyt)^{1/4}}, \quad x,y,t\in (0,\infty )\mbox{ and } \xi \geq 1.
\end{equation}

Let now $f\in L^p((0,\infty ),\mathbb{B})$. Let us denote by $\tilde{f}$ the extension of $f$ to $\mathbb{R}$ which satisfies $\tilde{f}(x)=0$, $x\leq 0$.
By defining $\mathcal{G}_{\mathcal{H}/2,\mathbb{B}}^\ell (\tilde{f})$ in the obvious way, we have that
\begin{align}\label{gdif}
    & \|\mathcal{G}_{\mathcal{L}_\alpha ,\mathbb{B}}^\ell (f)(x,\cdot )-\mathcal{G}_{\mathcal{H}/2,\mathbb{B} }^\ell(\tilde{f})(x,\cdot )\|_{L^2\big((0,\infty ),\frac{dt}{t};\mathbb{B}\big)}\nonumber\\
        & \qquad \qquad \leq\int_{(0,x/2)\cup (2x,\infty )}\|f(y)\|_\mathbb{B}\left(\Big\|t^\ell\partial_t^\ell W_t^{\mathcal{L}_\alpha }(x,y)\Big\|_H+\Big\|t^\ell\partial_t^\ell W_t^{\mathcal{H}/2}(x,y)\Big\|_H\right)dy\nonumber\\
        & \qquad \qquad \qquad + \int_{x/2}^{2x}\|f(y)\|_\B\Big\|t^\ell\partial_t^\ell \left[W_t^{\mathcal{L}_\alpha }(x,y)-W_t^{\mathcal{H}/2}(x,y)\right]\Big\|_Hdy\nonumber\\
        & \qquad \qquad = T_1(f)(x)+T_2(f)(x),\quad x\in (0,\infty ).
\end{align}

By using \eqref{2.4}, \eqref{B3menor} and \eqref{B3mayor} we obtain, when $x\in (0,\infty )$ and $y\in (0,x/2)\cup (2x,\infty)$,
\begin{align*}
    & \Big\|t^\ell\partial_t^\ell W_t^{\mathcal{L}_\alpha }(x,y)\Big\|_H+\Big\|t^\ell\partial_t^\ell W_t^{\mathcal{H}/2}(x,y)\Big\|_H
        \leq C\left(1+\frac{xy}{|x-y|^2}\right)\left(\int_0^\infty \frac{e^{-c|x-y|^2/t}}{t^2}dt\right)^{1/2}\\
    & \qquad \qquad \leq C\left(\int_0^\infty \frac{e^{-c|x-y|^2/t}}{t^2}dt\right)^{1/2}\leq C\frac{1}{|x-y|}
    \leq C\left\{\begin{array}{ll}
            \displaystyle \frac{1}{x},&\displaystyle 0<y<\frac{x}{2}\\
            &\\
            \displaystyle \frac{1}{y},&\displaystyle 0<2x<y
            \end{array}
    ,\right.
\end{align*}
because $|x-y|\sim x$, when $y\in (0,x/2)$ and $|x-y|\sim y$, when $y\in (2x,\infty)$.

Hence,
\begin{equation}\label{T1}
    T_1(f)(x)
        \leq C \left[H_0(\|f\|_\B))(x)+H_\infty (\|f\|_\B)(x)\right]<\infty ,\quad x\in (0,\infty ),
\end{equation}
where $H_0$ and $H_\infty$ represents the classical Hardy operators given by
$$H_0(g)(x)
    =\frac{1}{x}\int_0^xg(y)dy\qquad \mbox{ and }\qquad H_\infty (g)(x)=\int_x^\infty \frac{g(y)}{y}dy,\quad x\in (0,\infty ).$$

On the other hand,  by taking into account \eqref{B3bisH}, \eqref{B3bis} and \eqref{B4}, we can write,
\begin{align*}
    & \Big\|t^\ell\partial_t^\ell [W_t^{\mathcal{L}_\alpha }(x,y)-W_t^{\mathcal{H}/2}(x,y)]\Big\|_H
         \leq \left\{\left(\int_{0, \ \xi\leq 1}^\infty +\int_{0, \ \xi\geq 1}^\infty \right)\left|t^\ell\partial_t^\ell [W_t^{\mathcal{L}_\alpha }(x,y)-W_t^{\mathcal{H}/2}(x,y)]\right|^2\frac{dt}{t}\right\}^{1/2}\\
    & \qquad \qquad \leq C\left(\int_0^\infty \frac{e^{-c(x^2+y^2)/t}}{t^2}dt+\frac{1}{(xy)^{1/2}}\int_0^\infty \frac{e^{-c|x-y|^2/t}}{t^{3/2}}dt\right)^{1/2}\\
    & \qquad \qquad \leq C\left(\frac{1}{\sqrt{x^2+y^2}}+\frac{1}{(xy)^{1/4}\sqrt{|x-y|}}\right)\leq \frac{C}{y}\left(1+\sqrt{\frac{y}{|x-y|}}\right),
    \quad 0<\frac{x}{2}<y<2x.
\end{align*}
Hence,
\begin{equation}\label{T2}
    T_2(f)(x)\leq C\mathcal{N}(\|f\|_\mathbb{B})(x)<\infty ,\quad x\in (0,\infty ),
\end{equation}
where
$$\mathcal{N}(g)(x)
    =\int_{x/2}^{2x}\frac{1}{y}\left(1+\sqrt{\frac{y}{|x-y|}}\right)g(y)dy,\quad x\in (0,\infty ).$$

From \eqref{gdif}, \eqref{T1} and \eqref{T2} we conclude that, for every $x\in (0,\infty )$,
$$\|\mathcal{G}_{\mathcal{L}_\alpha ,\mathbb{B}}^\ell (f)(x,\cdot )-\mathcal{G}_{\mathcal{H}/2 ,\mathbb{B}}^\ell(\tilde{f})(x,\cdot )\|_{\gamma (H,\mathbb{B})}
    \leq C \left[H_0(\|f\|_\mathbb{B})(x)+H_\infty (\|f\|_\mathbb{B})(x)+\mathcal{N}(\|f\|_\mathbb{B})(x)\right]
    <\infty.$$
It is well-known that the Hardy operators $H_0$ and $H_\infty$ are bounded from $L^p(0,\infty )$ into itself
(see \cite[p. 244, (9.9.1) and (9.9.2)]{HLP}).
Moreover, Jensen's inequality allows us to
show that the operator $\mathcal{N}$ is also bounded from $L^p(0,\infty )$ into itself. Hence,
$$\|\mathcal{G}_{\mathcal{L}_\alpha ,\mathbb{B}}^\ell (f)-\mathcal{G}_{\mathcal{H}/2 ,\mathbb{B}}^\ell(\tilde{f})\|_{L^p((0,\infty ), \gamma (H,\mathbb{B})}
    \leq C\|f\|_{L^p((0,\infty ),\mathbb{B})}.$$
Since, as it was seen in Section~\ref{sec:Hermite}, $(b)$ holds provided that $\mathbb{B}$ is UMD, it follows that
$$\|\mathcal{G}_{\mathcal{L}_\alpha ,\mathbb{B}}^\ell (f)\|_{L^p((0,\infty ), \gamma (H,\mathbb{B})}
    \leq C\|f\|_{L^p((0,\infty ),\mathbb{B})}.$$

Note that we can obtain results analogous to Propositions~\ref{polarization} and \ref{G1G2} for the operator $\mathcal{L}_\alpha$
instead of $\mathcal{H}$. The remainder of the proof of $(a)\Rightarrow (d)$ and the proof of $(d)\Rightarrow (a)$ can be made by proceeding as
in the proof of the corresponding properties in Section~\ref{sec:Hermite}.



\end{document}